\documentclass[12pt,leqno,twoside]{amsart}
\usepackage{amssymb,amsfonts,amsmath,amsthm, amstext}
\usepackage{latexsym}
\usepackage{verbatim}
\usepackage{graphicx}
\usepackage{subcaption}
\usepackage{pgf,tikz}
\usepackage{float}
\usepackage[normalem]{ulem}

\usepackage[latin1]{inputenc}
\usepackage[T1]{fontenc}
\usepackage[colorlinks=true, pdfstartview=FitV, linkcolor=blue, citecolor=blue, urlcolor=blue]{hyperref}
\usepackage{amstext,amsmath,amscd, bezier,indentfirst,amsthm,amsgen,enumerate}
\usepackage[all,knot,arc,import,poly]{xy}
\usepackage{amsfonts,xcolor, soul}  
\usepackage{graphicx}
\usepackage{srcltx}
\usepackage{enumitem}
\usepackage{epsfig}

\usepackage{tikz}
\usepackage{mathrsfs}
\usepackage[margin=1in]{geometry}
\usepackage[normalem]{ulem}

\newtheorem{theorem}{Theorem}[section]
\newtheorem{corollary}[theorem]{Corollary}
\newtheorem{lemma}[theorem]{Lemma}

\theoremstyle{definition}
\newtheorem{definition}[theorem]{\sc Definition}
\theoremstyle{remark}
\newtheorem{remark}[theorem]{\sc Remark}
\theoremstyle{remark}

\theoremstyle{remark}
\newtheorem{example}[theorem]{\sc Example}
\theoremstyle{remark}

\theoremstyle{remark}
\newtheorem{convention}[theorem]{\sc Convention}

\theoremstyle{remark}

\newtheorem{lem/defi}[thm]{Lemma/Definition}
\newtheorem{defi/lem}[thm]{Definition/Lemma}
\newtheorem{prop/defi}[thm]{Proposition/Definition}

\theoremstyle{definition}

\theoremstyle{remark}

\newcommand{\hot}{\mathrm{h.o.t.}}

\newcommand{\im}{\mathop{\rm{Im}}\nolimits}
\newcommand{\mult}{{\rm{mult}}}

\newcommand{\ord}{\mathop{\rm{ord}}}

\newcommand{\grad}{\mathop{\rm{grad}}\nolimits}

\newcommand{\Horn}{{\rm{Horn}}}

\newcommand{\In}{\operatorname{in}}
\newcommand{\Supp}{\operatorname{supp}}

\newcommand{\GC}{\mathcal{GC}}

\newcommand{\Cone}{\mathrm{Cone}}

\newcommand{\be}{\begin{equation}}
\newcommand{\ee}{\end{equation}}

\newcommand{\beqn}{\begin{eqnarray*}}
\newcommand{\eeqn}{\end{eqnarray*}}

\newcommand{\cC}{{\mathcal C}}

\newcommand{\PP}{\mathcal{NP}}

\newcommand{\bC}{{\mathbb C}}

\newcommand{\F}{\mathbb{F}}

\newcommand{\bN}{{\mathbb N}}
\newcommand{\bZ}{{\mathbb Z}}

\newcommand{\bQ}{{\mathbb Q}}

\newcommand{\C}{{\mathbb C}}

\newcommand{\Q}{\mathbb{Q}}

\newcommand{\R}{\mathbb{R}}

\title{Bi-Lipschitz invariance of Newton polygons along gradient canyons}

\author{Piotr Migus}
\address{Air Force Institute of Technology,
ul. Ksi\c{e}cia Boles\l awa 6,
01-494 Warsaw, Poland}
\email{migus.piotr@gmail.com}

\author{Lauren\c tiu P\u aunescu}
\address{School of Mathematics and Statistics, University of Sydney,
  Sydney, NSW, 2006, Australia.}
\email{laurent@maths.usyd.edu.au}

\author{Mihai Tib\u ar\textsuperscript{*}}
\address{Math\' ematiques, UMR 8524 CNRS,
Universit\'e de Lille, \  59655 Villeneuve d'Ascq, France.}
\email{mtibar@univ-lille.fr}

\dedicatory{Dedicated to the memory of our friend and collaborator Mihai Tib\u ar.}

\begin{document}

\begingroup
\renewcommand{\thefootnote}{*}
\footnotetext{Mihai Tib\u ar sadly passed away before the completion of this manuscript.}
\endgroup

\maketitle

\begin{abstract}
	We study bi-Lipschitz right-equivalence of holomorphic function germs $f:(\C^2,0)\to(\C,0)$ via polar arcs and gradient canyons. For a polar arc $\gamma$ we consider the Newton polygon of $f_x(X+\gamma(Y),Y)$ and define its augmentation by adjoining the point $(0,\ord f(\gamma(y),y)-1)$. We prove that the resulting augmented Newton polygon is constant along each gradient canyon of degree $>1$ and is invariant under bi-Lipschitz right-equivalence. Moreover, its compact edges decompose into a topological part and a Lipschitz part: the latter encodes, through simple intercept relations, the second-level Henry--Parusi\'nski type invariants. As applications, we obtain two numerical bi-Lipschitz invariants attached to a canyon: its polar multiplicity and, via the Koike--Kuo--P\u aunescu curvature formula, the total asymptotic Gaussian curvature concentrated in it.
\end{abstract}

\section{Introduction}

Let $f,g:(\C^2,0)\to(\C,0)$ be holomorphic function germs and assume that $f=g\circ\varphi$, where $\varphi:(\C^2,0)\to(\C^2,0)$ is a bi-Lipschitz homeomorphism. Determining which analytic features of $f$ survive such an equivalence is a central problem in the Lipschitz geometry of complex singularities. Henry and Parusi\'nski \cite{HP} discovered that the bi-Lipschitz equivalence of analytic function germs admits continuous moduli; in particular this classification is strictly finer than the topological one. Continuous moduli exist also for real analytic germs, as shown in \cite{HP1}, the first paper where the observation that a bi-Lipschitz equivalence preserves the size of the gradient is used to construct bi-Lipschitz invariants. An alternative approach for quasihomogeneous germs, based on their one-variable height functions, is pursued in the recent works \cite{Al} (real polynomials, semialgebraic bi-Lipschitz equivalence) and \cite{CF} (complex case, where bi-Lipschitz and analytic equivalence of non-homogeneous quasihomogeneous germs coincide).

\medskip

In the 2--variable setting, polar curves and the behaviour of the gradient along holomorphic arcs provide an efficient language to organise such invariants. Building on the ``canyon'' viewpoint of \cite{PT} and on the subsequent clustering description of polar arcs and gradient canyons (and their mutual contacts) \cite{MPT}, one can associate to each gradient canyon $\cC$ a finite amount of discrete data, often referred to as an ``identity card'' (canyon degree, the order $h_{\cC}=\ord f(\gamma(y),y)$ and certain contact data inside appropriate clusters). More recently, a second-level family of higher Lipschitz invariants was introduced in \cite{MPT1}: given two polar arcs $\gamma,\gamma'$ with prescribed contact $\delta$ and equal order $\ord f(\gamma(y),y)=\ord f(\gamma'(y),y)$, one compares their normalised expansions and extracts an exponent $H$ (together with a coefficient difference at order $H$) which is bi-Lipschitz invariant under a natural inequality $H<h+\delta-1$ (Theorem~\ref{c:invar1}). 

\medskip

The goal of the present paper is to provide a Newton polygon interpretation of these higher invariants, and to package them into a single finite combinatorial object associated to each gradient canyon. For a polar arc $\gamma$ we consider the Newton polygon of $f_x$ relative to $\gamma$, i.e.\ the Newton polygon of $f_x(X+\gamma(Y),Y)$. If $h(\gamma):=\ord f(\gamma(y),y)$, we define the augmented Newton polygon by adding the point $(0,h(\gamma)-1)$ and taking the convex hull (Definition~\ref{def:augNP}). This augmentation is crucial in our approach: it leads to a polygon which is well-defined on each gradient canyon and admits a bi-Lipschitz invariance statement (Theorem~\ref{newton_inv}).  Example~\ref{ex:newt1} illustrates that $\widehat{\PP}(f_x,\cC)$ may jump in a family, hence it can detect bi-Lipschitz non-triviality.

\medskip

\noindent\textbf{Main result.}
For a gradient canyon $\cC=\GC(\gamma_*)$ of gradient degree $d_{\cC}>1$, the augmented Newton polygon is well-defined on $\cC$ (Lemma~\ref{lem:augNP_canyon}) and is a bi-Lipschitz invariant:
\[
\widehat{\PP}(f_x,\cC)\ \text{depends only on the bi-Lipschitz class of }f
\qquad\text{(Theorem~\ref{newton_inv}).}
\]
Conceptually, $\widehat{\PP}(f_x,\cC)$ splits into a topological and a Lipschitz part. The edges of co-slope $<\delta_{top}$ (where $\delta_{top}$ is the co-slope of the top edge of $\PP(f,\gamma)$) are topological invariants (Lemma~\ref{lem:NP-below-deltatop}). The remaining compact edges, whose co-slopes lie in $(\delta_{top},d_{\cC})$, are governed by the second-level invariants of \cite{MPT1}: if $E_\delta$ is such an edge and $\omega$ is the $Y$--intercept of its supporting line, then there exists a further polar arc $\gamma'$ with $\ord(\gamma-\gamma')=\delta$ and the associated canonical exponent satisfies $H_{\min}=\omega+\delta$ (in particular $H_{\min}<h_{\cC}+\delta-1$), see Corollary~\ref{cor:Hmin_equals_omega_plus_delta}. In this sense, our main theorem gives a concrete Newton polygon avatar of the higher invariants, turning them into a finite combinatorial object attached to each canyon.

\medskip

\noindent\textbf{Numerical corollaries: multiplicity and curvature of a canyon.} We distinguish the bar multiplicity $\mult^{\mathrm{bar}}(\cC)$ (Definition~\ref{def:bar-mult-canyon}), the number of polar representatives sharing a given trunk below order $d$, from the polar multiplicity $\mult(\cC)$ of a canyon in the sense of \cite{PT} (Definition~\ref{def:PT-mult-canyon}). The number $\mult^{\mathrm{bar}}(\cC)$ is the horizontal length of the top compact edge of $\widehat{\PP}(f_x,\cC)$ (Lemma~\ref{lem:mu_equals_length}), hence is bi-Lipschitz invariant (Corollary~\ref{cor:bar-mult}). The polar multiplicity factors as $\mult(\cC)=b_\cC\cdot\mult^{\mathrm{bar}}(\cC)$, where $b_\cC$ is the number of trunks at the canyon scale; its invariance follows from the canyon-disk correspondence, and this gives the invariance of $\mult(\cC)$ (Corollary~\ref{cor:polar-mult}). Combining this with the curvature formula of Koike--Kuo--P\u aunescu \cite{KKP}, we also obtain that the total asymptotic Gaussian curvature concentrated in a finite-degree canyon is a bi-Lipschitz invariant (Corollary~\ref{cor:curvature}).

\medskip

\noindent\textbf{Examples.}
Example~\ref{ex:newt} shows that $\widehat{\PP}(f_x,\cC)$, while bi-Lipschitz invariant, is not a complete invariant: it may remain constant in a family which is nevertheless not bi-Lipschitz trivial, as detected by the second-level invariants of \cite{MPT1}. Example~\ref{ex:newt1} gives a complementary picture: in a simple family the augmented Newton polygons (and the canyon degrees they encode) distinguish the special fibre $t=0$ from the generic fibres, hence the family is not bi-Lipschitz right-trivial.

\medskip

\noindent\textbf{Organisation of the paper.} 
The paper is organised as follows. Section~\ref{s:prelim} recalls the needed background on Puiseux arcs, contact orders and gradient canyons. Section~\ref{s:newtonpoly} discusses Newton polygons relative to a polar arc and proves the topological stability of the part below $\delta_{top}$. Section~\ref{newton_interpret} proves the bi-Lipschitz invariance of the augmented Newton polygon. Section~\ref{s:multiplicity} introduces the bar multiplicity, relates it to the polar multiplicity of \cite{PT}, and proves the corresponding multiplicity invariance; Section~\ref{s:curvature} derives the curvature counterpart from the Koike--Kuo--P\u aunescu formula; Section~\ref{s:examples} contains the examples. A glossary of the main notation is given at the end of the paper.

\section{Preliminaries}\label{s:prelim}

 Let $f,g:(\bC^2,0)\to (\bC,0)$ be holomorphic function germs such that $f=g\circ \varphi$, where $\varphi:(\bC^2,0)\to (\bC^2,0)$ is a bi-Lipschitz homeomorphism. 

Consider a holomorphic map germ 
\[
\alpha :(\bC,0)\to(\bC^2,0), \quad \alpha(t)=(z(t),w(t)) \not \equiv 0.
\]
The image set germ $\alpha_*=\im(\alpha)$ is a \emph{curve germ} at $0\in \bC^2$, also called a \emph{holomorphic arc} at $0$. There is a well-defined tangent line $T(\alpha_*)$ at $0$,   $T(\alpha_*)\in\bC P^1$. 

In order to introduce our results, we need to recall here several notations and facts from \cite{PT} and \cite{MPT}.

\medskip

Let $\F$ be the field of convergent fractional power series in an indeterminate $y$. By the Newton--Puiseux Theorem we have that $\F$ is algebraically closed, cf. \cite{BK}, \cite{Wa}.

A non-zero element of $\F$ has the form 
\begin{equation}\label{puis}
\eta(y)=a_0y^{n_0/N}+a_1y^{n_1/N}+a_2y^{n_2/N}+\cdots,\quad n_0<n_1<n_2<\cdots,
\end{equation}
where $a_i \in \bC^*$ and  $N,n_i\in \bN$ with $\gcd(N,n_0,n_1,\dots)=1$, $\lim \sup |a_i|^{\frac{1}{n_i}}<\infty$. The elements of $\F$ are called \emph{Puiseux arcs}. There are $N$ \emph{conjugates} of $\eta$ (including $\eta$ itself), which are the Puiseux arcs of the form
\[
\eta^{(k)}_{conj}(y):=\sum_{i\ge 0} a_i\varepsilon^{kn_i}y^{n_i/N}, \quad \varepsilon:=e^{\frac{2\pi\sqrt{-1}}{N}},
\]
where $k\in\{0,\dots ,N-1\}$.

By the \emph{order of a Puiseux arc} \eqref{puis} we mean $\ord \eta(y):=\frac{n_0}{N}$, and by the \emph{Puiseux multiplicity} we mean $m_{puiseux}(\eta)=N$, cf \cite{BK}, \cite{Wa}. 

Let $\F_1:=\{\eta \in \F \mid \ord \eta (y)\geq 1\}$.   
For any $\eta \in \F$ with $\ord \eta (y)>0$, the following map germ:
\[
\eta_{par} :(\bC,0)\to(\bC^2,0), \quad t\mapsto (\eta(t^N),t^N), \quad N:=m_{puiseux}(\eta),
\]
is holomorphic and injective (the latter since $\gcd(N,n_0,n_1,\dots)=1$), and all the conjugates of $\eta$ lead to the same irreducible curve $\im\eta_{par}$, which will be denoted by $\eta_*$.

\begin{definition}[Contact order of holomorphic arcs, cf.\ \cite{BK}, \cite{Wa}]\ \label{def:contact} \\
	The \emph{contact order} between two different holomorphic arcs $\alpha_*$ and $\beta_*$, where $\alpha(y),\beta(y) \in \F_1$ are Puiseux series representing these arcs, is defined as
	\[
	\max_{i,j} \ord\bigl(\alpha^{(i)}_{\mathrm{conj}}(y) - \beta^{(j)}_{\mathrm{conj}}(y)\bigr),
	\]
	where $\alpha^{(i)}_{\mathrm{conj}}$, $\beta^{(j)}_{\mathrm{conj}}$ run over all conjugates of $\alpha$ and $\beta$,
	respectively.
\end{definition}

\medskip

Let $f:(\bC^2,0)\rightarrow(\bC,0)$ be a holomorphic function germ, and let $m:= \ord_{0}f$. We say that $f$ is \emph{mini-regular in $x$ of order $m$},  if the initial form of the Taylor expansion of $f$ is not equal to $0$ at the point $(1,0)$, in other words $f_m(1,0)\neq 0$ where $f(x,y)=f_{m}(x,y)+f_{m+1}(x,y)+\hot$ is the homogeneous Taylor expansion of $f$.

From now on, we assume that $f$ is mini-regular in $x$. Then $f_x$ is mini-regular in $x$ of order $m-1$, since its initial form $\partial f_m/\partial x$ satisfies $(\partial f_m/\partial x)(1,0)=m\,f_m(1,0)\neq 0$. We then have the following Puiseux factorisations of $f$, and of its derivative $f_x$ with respect to $x$:
\[
f(x,y)=u\cdot \prod_{i=1}^m(x-\zeta_i(y)), \quad f_x(x,y)=v\cdot \prod_{i=1}^{m-1}(x-\gamma_i(y)),
\]
where $u,v$ are units and, by mini-regularity, $\zeta_i,\gamma_i\in\F_1$. \\ 
 If $f_x=g_1^{q_1} \cdots g_p^{q_p}$ is the decomposition into irreducible factors, then for each $i$ such that $g_i$ is not a factor of $f$, the curve $\{g_i=0\}$ is an irreducible component of the polar curve. In particular, there exists at least one Puiseux root $\gamma$ of $f_x$ such that $g_i(\gamma(y),y)\equiv 0$.

 \medskip
 
 \begin{definition}[Polar arcs and their order]\label{def:polar-arc}
 	A \emph{polar arc} of $f$ is a Puiseux root $\gamma\in\F_1$ of $f_x$ that is not a root of $f$; that is, $f_x(\gamma(y),y)\equiv 0$ and $f(\gamma(y),y)\not\equiv 0$. Its image curve germ $\gamma_*$ (a holomorphic arc, in the sense recalled above) is also called a polar arc when this causes no confusion. Following \cite{PT}, whenever we speak of an ``arc'' we always mean a Puiseux arc.
 	
 	For a polar arc $\gamma$, the \emph{order of $f$ along $\gamma$} is
 	\[
 	h(\gamma):=\ord f(\gamma(y),y)\in\Q_{>0}.
 	\]
 \end{definition}

\medskip

\subsection{Gradient degree and gradient canyon.}\label{ss:degree}
Let $\gamma$ be a polar arc of $f$, thus such that $f(\gamma(y),y)\not\equiv 0$ and $f_x(\gamma(y),y) \equiv 0$. The \emph{gradient degree} $d(\gamma)$ is the smallest number $q$ such that 
\[
\ord\bigl(\|\grad f(\gamma(y),y)\|\bigr)=\ord\bigl(\|\grad f(\gamma(y)+uy^q,y)\|\bigr),
\]
holds for generic $u\in \bC$. This number is rational, $d(\gamma)\in\Q$. The \emph{gradient canyon} $\GC(\gamma_*)$ is the subset of all curve germs $\alpha_*$, where $\alpha$ is a Puiseux arc of the form 
\[
\alpha(y):=\gamma +uy^{d(\gamma)}+\hot
\]
for any $u\in \bC$. By \cite[Proposition~3.7(b)]{PT} all the polars contained in the same canyon have the same gradient degree, and therefore we may speak about the degree of the canyon (or simply the canyon degree).

\begin{remark}\label{rem:puiseux-canyons}
	The definition above is geometric. When working with fixed Puiseux representatives, we shall nevertheless use the same word ``canyon''. This causes no ambiguity for contacts $\delta<d$: under the bi-Lipschitz correspondence, contacts between polars below the canyon degree are preserved, as follows from \cite[Lemma~3.1]{MPT1}.
\end{remark}

\medskip

\subsection{The Kuo-Lu tree.}\label{ss:kuolu-tree}
For the convenience of the reader we recall the tree-model of $f$ from \cite[\S4]{kuo-lu} (there denoted $M(f)$), since it is used repeatedly below; we denote it by $T(f)$. Write the Puiseux factorisation $f=u\prod_{i=1}^m(x-\zeta_i(y))$, with $u$ a unit, and call $\ord(\zeta_i-\zeta_j)$ (for $i\neq j$) the \emph{order of contact} of the roots $\zeta_i,\zeta_j$. The tree is built as follows. One starts from a vertical \emph{main trunk}, labelled by the multiplicity $m=\ord f$. On top of it one draws a horizontal \emph{bar} whose height is the smallest order of contact $b_1:=\min_{i\neq j}\ord(\zeta_i-\zeta_j)$; the roots are then split into the groups within which any two roots have contact $>b_1$, and each group is carried by a vertical \emph{trunk} rising from the bar (labelled by the number of roots it carries). One repeats this construction on every trunk, obtaining higher bars and thinner trunks, until each trunk carries a single root; such a trunk is a \emph{twig}. Thus tracing the path from the main trunk to a twig identifies a root $\zeta_i$. For two roots $\zeta_i,\zeta_j$, their order of contact $\ord(\zeta_i-\zeta_j)$ equals the height of the bar at which the corresponding two twigs separate, that is, the highest bar lying on \emph{both} paths from the main trunk to these twigs (equivalently, the lowest bar on the path joining the two twigs).

Two facts will be used. First, by \cite[\S4]{kuo-lu}, the orders of contact between the Puiseux roots of $f$ and those of its derivative $f_x$ (i.e.\ the polar arcs) are entirely determined by $T(f)$. More precisely, let $T(f_x)$ denote the tree-model of $f_x$, built as above from the Puiseux roots of $f_x$, and let $T^*(f_x)$ denote the auxiliary tree constructed out of $T(f)$ as follows (in \cite[\S4]{kuo-lu} these two trees are denoted by $M(f_x)$ and $M^*(f_x)$): if a bar $B$ of $T(f)$ carries $k$ trunks, of multiplicities $w_1,\dots,w_k$ (a trunk of multiplicity $1$ being a twig), one replaces each of them by a trunk of multiplicity $w_i-1$ (so that all twigs on $B$ are removed) and adds $k-1$ new twigs on $B$, the height of $B$ being unchanged (thus a bar carrying only two twigs disappears). The twigs of $T^*(f_x)$ are in one-to-one correspondence with those of $T(f_x)$, i.e.\ with the Puiseux roots of $f_x$ counted with multiplicity, and the order of contact between a twig of $T(f)$ and a twig of $T(f_x)$ equals that between the same twig of $T(f)$ and the corresponding twig of $T^*(f_x)$; the latter can be read off directly from the trees. Second, $T(f)$ is a topological invariant of the germ $f$, being determined by the characteristic exponents of the branches of $f$ together with their pairwise intersection multiplicities, which are classical topological invariants of plane curve germs, cf.\ \cite{BK}; more precisely, a topological equivalence $f=g\circ \phi$, with $\phi$ a homeomorphism of $(\C^2,0)$, induces an isomorphism $T(f)\cong T(g)$ matching the Puiseux roots of $f$ with those of $g$ and preserving all their orders of contact.

Following \cite{PT,MPT}, a bar from which polar arcs $\gamma$ with $\ord f(\gamma(y),y)=h$ depart is denoted by $B(h)$; here the rational number $h$, rather than the Kuo-Lu height of the bar, is used as the label of the bar. The value $h$ is itself a topological invariant of $f$.

\subsection{The order $\ord f(\gamma(y),y)$.}\label{ss:order-h}
The order $h=\ord f(\gamma(y),y)$ of $f$ along a polar arc $\gamma$ (Definition~\ref{def:polar-arc}) is constant in the same canyon $\cC$, see e.g.\ \cite[Proposition~3.7(c)]{PT}; we may therefore denote it by $h_{\cC}$.

\medskip

\subsection{Clusters of canyons.}\label{ss:clusters}
Let $Z(f):=f^{-1}(0)=\{f=0\}$ denote the zero locus of $f$. Let $G_\ell(f)$ be the subset of canyons tangent to the line $\ell$ of the tangent cone $\Cone_0(f)$ at the origin of the curve $Z(f)$. Let then $G_{\ell,d,B(h)}(f)\subset G_\ell(f)$ denote the union of gradient canyons of a fixed degree $d>1$, the  polars of which grow on the same bar $B(h)$. 

\medskip 	

We write
\[
G_{\ell,d,B(h)}(f) = \{ \GC_1(f), \dots, \GC_r(f) \}
\]
for the finite family of all gradient canyons in this cluster; thus the index $i$ in $\GC_i(f)$
simply labels the distinct canyons.

\medskip

\subsection{Contact of canyons.}\label{ss:contact}

A fixed gradient canyon $\GC_i(f) \in G_{\ell,d,B(h)}(f)$ has an order of contact $k(i,j)$ with any other gradient canyon $\GC_j(f) \in G_{\ell,d,B(h)}(f)$ of the same cluster. We define $k(i,j)$ to be the contact order between some polar arc in $\GC_i(f)$ and some polar arc in $\GC_j(f)$ in the sense of Definition~\ref{def:contact}. By \cite[\S5.3]{PT}, whenever $\GC_i(f)$ and $\GC_j(f)$ are distinct canyons of the same degree $d>1$, this contact order is strictly smaller than $d$ and hence does not depend on the choice of the polars; in particular $k(i,j)$ is well defined.

For $\kappa\in\Q$, the \emph{multiplicity} of $\kappa$ as a contact order of $\GC_i(f)$ is the number of canyons $\GC_j(f)$, $j\neq i$, of the cluster $G_{\ell,d,B(h)}(f)$ such that $k(i,j)=\kappa$.

\medskip

\subsection{Sub-clusters of canyons.}\label{ss:subclusters}
Let $K_{\ell,d,B(h),i}(f)$ denote the (un-ordered) set of those contact orders $k(i,j)$ of the fixed canyon $\GC_i(f)$, counted with multiplicity.

Finally, let $G_{\ell,d,B(h),\omega}(f)$ be the union of canyons from $G_{\ell,d,B(h)}$ which have exactly the same set $\omega=K_{\ell,d,B(h),i}(f)$ of  orders of contact $\ge 0$ with the other canyons from $G_{\ell,d,B(h)}$. We then have a partition:
\[
G_{\ell,d,B(h)}(f)= \bigsqcup_{\omega} G_{\ell,d,B(h),\omega}(f).
\]

\medskip

\subsection{Lipschitz invariants and ``identity card''.}\label{ss:lipinvar}
By {\cite[Theorem~5.9]{PT} and \cite[Theorem~4.3]{MPT}}, for any degree $d>1$, any bar $B(h)$,  and any rational $h$, the following are bi-Lipschitz invariants:
\begin{enumerate}
	\item the cluster of canyons $G_{\ell,d,B(h)}$,
	\item the set of contact orders $K_{\ell,d,B(h),i}(f)$, and for each such set, the sub-cluster of canyons $G_{\ell,d,B(h),K_{\ell,d,B(h),i}}(f)$,
	\item the contact orders between any two clusters of type $G_{\ell,d,B(h),K_{\ell,d,B(h),i}}(f)$.
\end{enumerate} 

\medskip

Let $\cC \in G_\ell(f)$ be a canyon, and let $\gamma$ be some polar arc in $\cC$, where $\ell$ is in the tangent cone of $Z(f)$. If $d_{\cC}$ denotes the gradient degree of $\cC$, then
\[
f(\gamma(y),y)=a_{h_{\cC}}y^{h_\cC}+\hot
\]
by \cite[Proposition~3.7.]{PT},  where $a_{h_{\cC}}$ and  $h_{\cC}$ depend only on the canyon. 

The main theorem of \cite{HP}, completed by \cite{PT} {(see \cite[Theorem~4.4]{MPT})},  tells that we also have the following bi-Lipschitz invariant:

\begin{enumerate}
	\item[(d)] The effect of the bi-Lipschitz map $\varphi$ on each such couple {$(h_{\cC}, a_{h_{\cC}})$} is the identity on {$h_{\cC}$}, and the multiplication of $a_{h_{\cC}}$ by $c^{- h_{\cC}}$, where $c\neq 0$ is a constant depending on $\varphi$ and $\ell$, but the same for all canyons $\cC \in G_\ell(f)$.
\end{enumerate}

We say that the above bunch of data, which yield the described bi-Lipschitz invariants, constitute the ``identity card'' of $f$.  
	
\medskip

\begin{convention}[The bi-Lipschitz correspondence of canyons]\label{conv:bilip}
	Throughout the paper, the phrase ``bi-Lipschitz invariant'' is understood as follows. Let $f=g\circ\varphi$, with $\varphi:(\C^2,0)\to(\C^2,0)$ a bi-Lipschitz homeomorphism and $f,g$ mini-regular in $x$. By \cite[Theorem~5.9]{PT} (built on \cite[Theorem~5.8]{PT}; see also \cite[Theorem~4.3]{MPT}), although $\varphi$ need not map a gradient canyon of $f$ onto a gradient canyon of $g$, it nevertheless induces a canonical bijection
	\[
	\Phi:\ \{\text{gradient canyons of }f\text{ of degree }>1\}\ \xrightarrow{\ \sim\ }\ \{\text{gradient canyons of }g\text{ of degree }>1\},
	\]
	obtained through the correspondence of canyon disks in the Milnor fibres. This bijection preserves the canyon degree, carries the cluster $G_{\ell,d,B(h)}(f)$ onto $G_{\ell,d,B(h)}(g)$ and the sub-cluster $G_{\ell,d,B(h),\omega}(f)$ onto $G_{\ell,d,B(h),\omega}(g)$, and preserves the orders of contact between canyons.
	
	Accordingly, a quantity $I_f(\cC)$ attached to each gradient canyon $\cC$ of $f$ of degree $>1$ is called a \emph{bi-Lipschitz invariant} if $I_f(\cC)=I_g\bigl(\Phi(\cC)\bigr)$ for every such $\cC$; a quantity $I_f(\cC,\cC')$ attached to a pair of canyons (with a prescribed order of contact) is called a bi-Lipschitz invariant if $I_f(\cC,\cC')=I_g\bigl(\Phi(\cC),\Phi(\cC')\bigr)$ for every such pair. For the leading data $(h_\cC,a_{h_\cC})$, which are defined only up to the global rescaling described in item~(d) above, the expression ``the effect of $\varphi$'' refers to that transformation rule (the identity on $h_\cC$, and multiplication of $a_{h_\cC}$ by $c^{-h_\cC}$, with $c\neq 0$ common to all canyons of $G_\ell(f)$); likewise, in the second-level statement below (Theorem~\ref{c:invar1}), ``the effect of $\varphi$ on the pair $(H,\tilde a-\tilde a')$'' means that the exponent $H=H(\cC,\cC')$ is preserved by $\Phi$ while $\tilde a-\tilde a'$ is multiplied by $c^{\,h-H}$.
	
	With this terminology, the statements of items~(a)--(d) above are to be read through $\Phi$; in particular, item~(c) means precisely that $\Phi$ preserves these contact orders. The same convention governs the later assertions that an augmented Newton polygon $\widehat{\PP}(f_x,\cC)$ (Theorem~\ref{newton_inv}), a second-level exponent $H_{\min}(\cC,\cC')$ (Corollary~\ref{cor:Hmin_equals_omega_plus_delta}), or a multiplicity of a canyon (Corollaries~\ref{cor:bar-mult} and~\ref{cor:polar-mult}) ``is a bi-Lipschitz invariant'': it takes equal values on $\cC$ and $\Phi(\cC)$, resp.\ on the pair $(\cC,\cC')$ and $(\Phi(\cC),\Phi(\cC'))$.
\end{convention}

\subsection{Higher invariants: second level}\label{ss:second-level}

In this section we recall the result from \cite[Corollary 3.5]{MPT1}.

We assume that there are at least two distinct gradient canyons contained in $G_\ell(f)$, $\GC(\gamma_{*})$ and $\GC(\gamma'_{*})$, of canyon degrees $d$ and $d'$, respectively. Let $\delta \ge 1$ denote the contact order between them.

In order to state the theorem, we now assume that the orders are equal: 
\[
\ord f(\gamma(y),y) = \ord f(\gamma'(y),y),
\]
 and we  introduce the following notation:

\begin{itemize}
\item $a$ is the coefficient of $y^{h}$ in the expansion of $f(\gamma(y),y)$,
\item $a'$ is the coefficient of $y^{h}$ in the expansion of $f(\gamma'(y),y)$, 
\item $H=\ord \bigl[\frac{1}{a}f(\gamma(y),y)-\frac{1}{a'}f(\gamma'(y),y) \bigr]$,
\item $\tilde{a}$ is the coefficient of $y^{H}$ in the expansion of $\frac{1}{a}f(\gamma(y),y)$,
\item $\tilde{a}'$ is the coefficient of $y^{H}$ in the expansion of $\frac{1}{a'}f(\gamma'(y),y)$.
\end{itemize}

With these notations, we have:
\begin{theorem}\label{c:invar1}
Let $f,g :(\bC^2,0)\to (\bC,0)$ be holomorphic functions such that $f=g\circ \varphi$, where $\varphi:(\bC^2,0)\to (\bC^2,0)$ is a bi-Lipschitz homeomorphism and $f$ is mini-regular in $x$. Let $h=\ord f(\gamma(y),y) = \ord f(\gamma'(y),y)$.
If 
 \begin{equation}\label{eq:H-cond}
  H<  h+\delta-1,
\end{equation}
 then the effect  of the bi-Lipschitz map $\varphi$ on the pair $\bigl( H, (\tilde{a}-\tilde{a}')\bigr)$ is the identity on $H$, and $(\tilde{a}-\tilde{a}')$ is multiplied by $c^{h-H}$.
\end{theorem}

\section{Newton polygon and related results}\label{s:newtonpoly}
Let $f :(\bC^2,0)\to (\bC,0)$ be a holomorphic function germ, and let $\eta(y)=a_0y^{n_0/N}+a_1y^{n_1/N}+a_2y^{n_2/N}+\cdots,\quad n_0<n_1<n_2<\cdots$, where $a_i \in \bC^*$ and  $N,n_i\in \bN$ with $\gcd(N,n_0,n_1,\dots)=1$, $\limsup |a_i|^{\frac{1}{n_i}}<\infty$. The Newton polygon $\PP(f)$ is defined in the usual way, as follows. Let us write
\[
f(x, y):=\sum c_{ij}x^i y^{j}, \quad c_{ij}\neq 0, \quad (i,j)\in \bZ_{\ge 0}\times\bZ_{\ge 0}.
\]
A monomial term with $c_{ij}\neq 0$ is represented by a ``Newton dot'' at $(i,j)$. We shall simply call it a dot of $f$. The boundary of the convex hull generated by $\{ (i+u,j+v)\mid u,v\geq 0 \}$, for all dots $(i,j)$, is the Newton polygon $\PP(f)$.

Now let us change variables (formally):
\[
X:=x-\eta(y), \quad Y:=y,\quad F(X,Y):=f(X+\eta(Y),Y).
\]
The Newton polygon $\PP(F)$ is called the Newton polygon of $f$ relative to $\eta$, denoted by $\PP(f,\eta)$; this is the Newton polygon relative to an arc in the sense of Kuo and Parusi\'nski \cite{KP}. Note that for Puiseux $\eta$ the series $F(X,Y)$ may involve fractional powers of $Y$; we use the same definition of the Newton polygon for such series, allowing exponents in $\bQ_{\ge 0}$.

For a compact edge of a Newton polygon or, more generally, for a line in $\R^2$ through $(u,0)$ and $(0,v)$ with $u,v>0$, we call $v/u$ its \emph{co-slope}. For a weight $\nu=[a,b]$ with $a,b>0$ we set $\nu(X^iY^j):=ai+bj$ and, for a (possibly fractional) series $G=\sum c_{ij}X^iY^j$, we write $\ord_\nu(G):=\min\{\nu(X^iY^j)\mid c_{ij}\neq 0\}$ and denote by $\In_\nu(G)$ its \emph{$\nu$-initial form}, namely the sum of the terms attaining this minimum. The compact edge of $\PP(G)$ exposed by the weight $\nu=[a,b]$ has co-slope $a/b$. We write $\Supp(H):=\{(i,j)\in\mathbb Z_{\ge0}\times\mathbb Q_{\ge0}\mid c_{ij}\neq 0\}$ for a (possibly fractional) series $H(X,Y)=\sum c_{ij}X^iY^j$.

\subsection{Gradient degree}
Next, we will need a construction from \cite{PT}. In this subsection we assume that $\eta=\gamma$ is a polar arc of $f$, and we keep the notation $F(X,Y)=f(X+\gamma(Y),Y)$. Set $h:=\ord_Y F(0,Y)=\ord f(\gamma(y),y)$. Let $E_{top}$ denote the edge of $\PP(F)$ whose left vertex is $(0,h)$ and whose right vertex is $(m_{top}, q_{top})$. We call it the \emph{top edge} and denote its co-slope by $\delta_{top}$. Since $f(\gamma(y),y)\not \equiv 0$ and $f_x(\gamma(y),y)\equiv 0$, the top edge $E_{top}$ has its left endpoint on the line $X=0$ and there are no Newton dots on the vertical line $X=1$.

\medskip

Let $(\hat{m}_{top},\hat{q}_{top})\neq (0,h)$ be the dot of $F$ on $E_{top}$ which is closest to the left end of $E_{top}$ (which is $(0,h)$). Of course $(\hat{m}_{top},\hat{q}_{top})$ may coincide with $(m_{top},q_{top})$. Then, clearly,
\[
2\leq \hat{m}_{top}\leq m_{top}, \quad \frac{h-\hat{q}_{top}}{\hat{m}_{top}}=\frac{h-q_{top}}{m_{top}}=\delta_{top}.
\]

Now we draw a line $L$ through $(1, h-1)$ with the following two properties:
\begin{itemize}
\item if $(m',q')$ is a dot of $F_X$, then $(m'+1,q')$ lies on or above $L$;
\item there exists a dot $(m^*, q^*)$ of $F_X$ such that $(m^*+1,q^*)\in L$. (Of course, $(m^*+1, q^*)$ may coincide with $(\hat{m}_{top}, \hat{q}_{top})$.)
\end{itemize}
In \cite[Lemma~3.4(i)]{PT}, it is shown that the co-slope of $L$ is equal to the degree $d:=d(\gamma)$. The above construction is shown in Figure~\ref{fig:3} (note that in this case $(\hat{m}_{top},\hat{q}_{top})=(m_{top},q_{top})$).

\medskip

\begin{figure}[h]
\centering
\begin{tikzpicture}[scale=0.8]
    \draw[->] (-0.5, 0) -- (5, 0) ;
    \draw[->] (0, -0.5) -- (0, 8.8) ;
    
    \foreach \x in {1,2,3,4} {
        \draw[gray, dashed] (\x, 0) -- (\x, 8.3);
    }
    \foreach \y in {1,2,3,4,5,6,7,8} {
        \draw[gray, dashed] (0, \y) -- (4.5, \y);
    }

    \draw[thick] (0, 7.5) -- (3, 0) node[pos=0.5, below left] {$E_{top}$};
    \draw[thick] (2, 3) --(3,0);
    \draw[thick, dotted] (2, 3) --(1,6.5) node[pos=0.5, above right] {$L$};
    \fill[black] (0, 7.5) circle (3pt) node[left] {$(0,h)$};
    \fill[black] (1, 6.5) circle (3pt) node[right] {$(1,h-1)$};
    \fill[black] (2, 3) circle (3pt) node[right]{$(m^*+1,q^*)$};
    \fill[black] (3, 0) circle (3pt) node[below] {$(\hat{m}_{top}, \hat{q}_{top})$};

\end{tikzpicture}
\caption{Construction of the line $L$ whose co-slope equals the gradient degree.}
    \label{fig:3}
\end{figure}
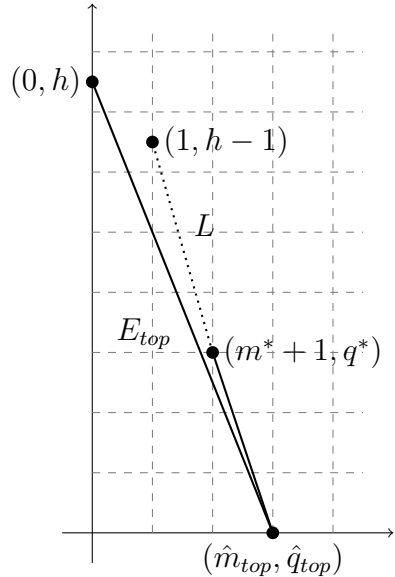

\subsection{Auxiliary lemma}

\begin{lemma}\label{lem:NP-below-deltatop}
	Let $f:(\C^2,0)\to(\C,0)$ be mini-regular in $x$, and let $\gamma$ be a polar arc of $f$. Let $\delta_{top}$ denote the co-slope of the top edge of $\PP(f,\gamma)$. Then the part of $\PP(f,\gamma)$ consisting of all compact edges whose co-slope is $<\delta_{top}$ (with their co-slopes and horizontal lengths) is a topological invariant of $f$ (i.e.\ it depends only on the embedded topological type of $f$).
	\footnote{In the real-analytic category, a bi-Lipschitz analogue of this statement, the invariance of the compact edges of the Newton polygon relative to an arc (the \emph{Newton boundary} in the terminology of \cite{KoP}), was proved by Koike and Parusi\'nski \cite[Corollary~2.7]{KoP}; their argument does not depend on the ground field and hence applies over $\C$ as well. We thank A.~Parusi\'nski for this observation.}
	
\end{lemma}

\begin{proof}
	Write the Puiseux factorisation
	\[
	f(x,y)=u(x,y)\prod_{i=1}^m\bigl(x-\zeta_i(y)\bigr),
	\]
	where $u$ is a unit and the $\zeta_i$ are Puiseux roots of $f$.  Set
	\[
	F(X,Y):=f\bigl(X+\gamma(Y),Y\bigr)
	=u(X+\gamma(Y),Y)\prod_{i=1}^m\bigl(X+\gamma(Y)-\zeta_i(Y)\bigr),
	\]
	so that $\PP(f,\gamma)$ is the Newton polygon of $F$.
	
	\medskip\noindent
	\emph{Step 1: contacts with a root under the polar.}
	By the description of polar arcs in terms of the Kuo-Lu tree, there exists a Puiseux root $\alpha$ of $f$ such that $\gamma$ grows from the trunk of $\alpha$ at a bar of height $\delta_{top}$; in particular
	\[
	\ord\bigl(\gamma(y)-\alpha(y)\bigr)=\delta_{top},
	\]
	and $\gamma$ and $\alpha$ have the same Puiseux coefficients in all terms of order $<\delta_{top}$, cf.\ \cite{kuo-lu}.  
	
	Fix a common denominator for the Puiseux series of $\alpha$, $\gamma$ and all $\zeta_i$.  For any Puiseux root $\zeta$ of $f$ with
	\[
	c:=\ord\bigl(\alpha(y)-\zeta(y)\bigr)<\delta_{top},
	\]
	we then have
	\begin{equation}\label{eq:contact-stability}
	\ord\bigl(\gamma(y)-\zeta(y)\bigr)
	=\ord\bigl(\alpha(y)-\zeta(y)\bigr)=c.
	\end{equation}
	Indeed, the difference $\gamma-\alpha$ has order exactly $\delta_{top}$, so all terms of order $<\delta_{top}$ coincide in $\gamma$ and $\alpha$, and the leading term of $\alpha-\zeta$ is of order $c<\delta_{top}$ and is not affected by adding $\gamma-\alpha$.
	
	Denote
	\[
	c_i:=\ord\bigl(\gamma(y)-\zeta_i(y)\bigr),\qquad i=1,\dots,m.
	\]
	By \eqref{eq:contact-stability}, the sub-multiset $\{c_i\mid c_i<\delta_{top}\}$ coincides with the sub-multiset
	\[
	\bigl\{\ord\bigl(\alpha-\zeta_i\bigr):\ \ord(\alpha-\zeta_i)<\delta_{top}\bigr\}.
	\]
	
	\medskip\noindent
	\emph{Step 2: the part of $\PP(f,\gamma)$ below $E_{top}$ depends only on the $c_i<\delta_{top}$.}
	Write
	\[
	H_i(Y):=\zeta_i(Y)-\gamma(Y), \qquad \ord H_i=c_i.
	\]
	Then
	\[
	F(X,Y)=u(X+\gamma(Y),Y)\prod_{i=1}^m\bigl(X-H_i(Y)\bigr).
	\]	
	The Newton polygon of a product is the Minkowski sum of the Newton polygons of the factors.  Since $u$ is a unit, its Newton polygon consists of the single point $(0,0)$, so taking the Minkowski sum with $\PP(u)$ does not change the Newton polygon.  In particular, we may ignore the factor $u(X+\gamma(Y),Y)$ when computing $\PP(f,\gamma)$, and $\PP(f,\gamma)$ is completely determined by the factors $(X-H_i(Y))$.

	For each factor $X-H_i(Y)$ the only Newton dots that can lie on the lower boundary are $(1,0)$ (coming from $X$) and $(0,c_i)$ (coming from the leading term of $H_i$); all other dots lie strictly above the segment joining these two points.  Hence the Newton polygon of $X-H_i(Y)$ has a unique compact edge of co-slope $c_i$ and horizontal length~$1$.
	
	Consequently, the Newton polygon $\PP(f,\gamma)$ is the Minkowski sum of these segments, and its compact edges (with their co-slopes and horizontal lengths) are determined solely by the multiset $\{c_i\}_{i=1}^m$.  Moreover, any edge of co-slope $<\delta_{top}$ arises from those $c_i$ which are $<\delta_{top}$, and the total horizontal length of all edges of a fixed co-slope $q<\delta_{top}$ equals the number of indices $i$ with $c_i=q$.
	
	Thus the part of $\PP(f,\gamma)$ consisting of edges of co-slope $<\delta_{top}$ is completely determined by the multiset $\{c_i\mid c_i<\delta_{top}\}$.
	
	\medskip\noindent
	\emph{Step 3: topological invariance of the contacts.}
	By Step~2, the configuration of compact edges of $\PP(f,\gamma)$ with co-slope $<\delta_{top}$, together with their co-slopes and horizontal lengths, is completely determined by the multiset
	\[
	\{c_i\mid c_i<\delta_{top}\}
	=\Bigl\{\ord\bigl(\gamma(y)-\zeta_i(y)\bigr)\mid \ord\bigl(\gamma(y)-\zeta_i(y)\bigr)<\delta_{top}\Bigr\}.
	\]
	By Step~1, for every $c_i<\delta_{top}$ we have
	\[
	c_i=\ord\bigl(\gamma(y)-\zeta_i(y)\bigr)
	=\ord\bigl(\alpha(y)-\zeta_i(y)\bigr),
	\]
	where $\alpha$ is a Puiseux root of $f$ such that $\gamma$ grows from the trunk of $\alpha$ at a bar of height $\delta_{top}$ in the Kuo-Lu tree (see \cite{kuo-lu}). Hence the above multiset coincides with
	\[
	\Bigl\{\ord\bigl(\alpha(y)-\zeta_i(y)\bigr)\mid \ord\bigl(\alpha(y)-\zeta_i(y)\bigr)<\delta_{top}\Bigr\}.
	\]
	
	These orders of contact are encoded in the Kuo-Lu tree $T(f)$: for two Puiseux roots $\alpha,\zeta_i$, the order of contact $\ord(\alpha-\zeta_i)$ equals the height of the bar at which the corresponding twigs separate in $T(f)$, cf.\ Subsection~\ref{ss:kuolu-tree}. In particular, the multiset
	\[
	\Bigl\{\ord\bigl(\alpha(y)-\zeta_i(y)\bigr)\mid \ord\bigl(\alpha(y)-\zeta_i(y)\bigr)<\delta_{top}\Bigr\}
	\]
	can be completely read off from the tree $T(f)$.
	
	On the other hand, as shown in \cite{kuo-lu}, the Kuo-Lu tree $T(f)$ is a topological invariant of the function germ $f$.  In particular, the above multiset of contact orders, and hence the multiset $\{c_i\mid c_i<\delta_{top}\}$, are topological invariants of $f$.
	
	\medskip
	
	Combining Step~2 and Step~3, we obtain that the part of the Newton polygon $\PP(f,\gamma)$ consisting of edges with co-slope $<\delta_{top}$, together with their co-slopes and horizontal lengths, is a topological invariant of $f$. This proves Lemma~\ref{lem:NP-below-deltatop}.	
\end{proof}

\begin{remark}\label{rem:topcorrespondence}
	Lemma~\ref{lem:NP-below-deltatop} has the following meaning for a topological equivalence $f=g\circ\phi$, with $\phi$ a homeomorphism. The induced isomorphism $T(f)\cong T(g)$ preserves $\delta_{top}$ and matches $\gamma$ with a polar arc $\gamma'$ of $g$ whose part of $\PP(g,\gamma')$ below $\delta_{top}$ coincides with that of $\PP(f,\gamma)$. This matching is purely combinatorial: $\phi$ need not send polar arcs to polar arcs, but the twig of $T^*(f_x)$ carrying $\gamma$ corresponds to a twig of $T^*(g_x)$, and any polar of $g$ on it may serve as $\gamma'$.
\end{remark}

\section{The impact of the bi-Lipschitz equivalence on the Newton polygon $\PP(f_x,\gamma)$}\label{newton_interpret}

In this section we describe the consequences of Theorem~\ref{c:invar1}. We adopt the notation introduced in Subsection~\ref{ss:second-level}.

\begin{definition}\label{def:augNP}
	Let $f:(\C^2,0)\to(\C,0)$ be mini-regular in $x$ and let $\gamma$ be a polar arc of $f$. Set $h(\gamma):=\ord f(\gamma(y),y)$. We define the \emph{augmented Newton polygon of $f_x$ along $\gamma$} by
	\[
	\widehat{\PP}(f_x,\gamma)
	:=\operatorname{Conv}\Bigl(\PP(f_x,\gamma)\cup\{(0,h(\gamma)-1)\}\Bigr)\subset\R^2.
	\]
\end{definition}

\medskip

Directly from the definition of a gradient canyon and from \cite[Proposition~3.7(b)]{PT}, we obtain the following.

\medskip 

\begin{lemma}[Well-definedness on a gradient canyon]\label{lem:augNP_canyon}
	Let $f:(\C^2,0)\to(\C,0)$ be mini-regular in $x$, and let $\cC=\GC(\gamma_*)$ be a gradient canyon of degree $d:=d_{\cC}$. If $\gamma_1,\gamma_2$ are polar arcs contained in $\cC$, then
	\[
	\widehat{\PP}(f_x,\gamma_1)=\widehat{\PP}(f_x,\gamma_2).
	\]
	In particular, $\widehat{\PP}(f_x,\gamma)$ depends only on the canyon $\cC$, and we may denote it by $\widehat{\PP}(f_x,\cC)$.
\end{lemma}

\medskip

\begin{lemma}\label{lem:h-and-a-stable}
	Let $f:(\C^2,0)\to(\C,0)$ be mini-regular in $x$, and let $\gamma,\gamma'$ be polar arcs of $f$. Let $\delta_{top}$ be the co-slope of the top edge $E_{top}$ of $\PP(f,\gamma)$. Assume that
	\[
	\delta:=\ord\bigl(\gamma(y)-\gamma'(y)\bigr)>\delta_{top}.
	\]
	Write
	\[
	f(\gamma(y),y)=a\,y^{h}+\hot,\qquad
	f(\gamma'(y),y)=a'\,y^{h'}+\hot,
	\]
	with $a,a'\in\C^*$ and $h,h'\in\Q_{>0}$. Then $h'=h$ and $a'=a$.
\end{lemma}

\begin{proof}
	Since $f$ is mini-regular in $x$, we may write the Puiseux factorisation
	\[
	f(x,y)=u(x,y)\prod_{j=1}^{m}\bigl(x-\zeta_j(y)\bigr),
	\]
	where $u$ is a unit and the $\zeta_j$ are the Puiseux roots of $f$ (counted with multiplicity). Set
	\[
	c_j:=\ord\bigl(\gamma(y)-\zeta_j(y)\bigr).
	\]
	Since $\delta_{top}$ is the co-slope of the top edge of $\PP(f,\gamma)$, we have $\max_j c_j=\delta_{top}$ (cf.\ Section~\ref{s:newtonpoly}).

	Now write $\gamma'(y)=\gamma(y)+t(y)$ with $\ord t=\delta>\delta_{top}$. Fix $j\in\{1,\dots,m\}$ and consider
	\[
	\gamma'(y)-\zeta_j(y)=\bigl(\gamma(y)-\zeta_j(y)\bigr)+t(y).
	\]
	Since $\ord(\gamma-\zeta_j)=c_j\le \delta_{top}<\delta=\ord t$, adding $t(y)$ does not change either the order or the leading coefficient of the Puiseux arc $\gamma(y)-\zeta_j(y)$. Hence
	\[
	\ord\bigl(\gamma'(y)-\zeta_j(y)\bigr)=\ord\bigl(\gamma(y)-\zeta_j(y)\bigr)=c_j,
	\]
	and the leading coefficient of $\gamma'(y)-\zeta_j(y)$ equals that of $\gamma(y)-\zeta_j(y)$.
	
	Therefore,
	\[
	\ord f(\gamma'(y),y)
	=\ord u(\gamma'(y),y)+\sum_{j=1}^m \ord(\gamma'(y)-\zeta_j(y))
	=0+\sum_{j=1}^m c_j
	=\ord f(\gamma(y),y),
	\]
	so $h'=h$.
	
	Moreover, the coefficient of $y^{h}$ in $f(\gamma(y),y)$ is
	\[
	a=u(0,0)\cdot \prod_{j=1}^m \bigl(\text{leading coefficient of }(\gamma-\zeta_j)\bigr),
	\]
	and the analogous formula holds for $a'$ with $\gamma'$ in place of $\gamma$. Since the leading coefficients of all factors $(\gamma'-\zeta_j)$ coincide with those of $(\gamma-\zeta_j)$, we get $a'=a$.
\end{proof}

\medskip

\medskip
Given two gradient canyons $\cC=\GC(\gamma_*)$ and $\cC'=\GC(\gamma'_*)$ such that $\ord f(\gamma(y),y)=\ord f(\gamma'(y),y)=h$, set $\delta:=\ord(\gamma-\gamma')$. For polar arcs $\alpha_*\subset\cC$ and $\alpha'_*\subset\cC'$ with $\ord(\alpha-\alpha')=\delta$, let $H(\alpha,\alpha')$ be as in Subsection~\ref{ss:second-level}, and define
\[
H_{\min}(\cC,\cC')
:=\min\Bigl\{\,H(\alpha,\alpha')\ \Bigm|\ 
\alpha_*\subset\cC,\ \alpha'_*\subset\cC',\ \ord(\alpha-\alpha')=\delta\Bigr\}.
\]

\medskip

As a consequence of Theorem~\ref{c:invar1} and the Newton--Puiseux algorithm applied to an intermediate compact edge of $\PP(f_x,\gamma)$, the associated second-level exponent admits a canonical choice which can be read off from the vertical intercept of that edge.

\begin{corollary}\label{cor:Hmin_equals_omega_plus_delta}
	Let $f:(\C^2,0)\to(\C,0)$ be mini-regular in $x$ and let $\cC=\GC(\gamma_*)$ be a gradient canyon of degree $d_{\cC}>1$. Let $E_\delta$ be a compact edge of $\PP(f_x,\gamma)$ of co-slope $\delta$ with $\delta_{top}<\delta<d_{\cC}$, and let $\omega$ be the $Y$--intercept of its supporting line (equivalently, $\omega=\min\{q+\delta m\mid (m,q)\in\Supp(F_X)\}$ for $F(X,Y)=f(X+\gamma(Y),Y)$). Then there exists a polar arc $\gamma'$ with $\ord(\gamma-\gamma')=\delta$ and, writing $\cC':=\GC(\gamma'_*)$, one has
	\[
	H_{\min}(\cC,\cC')=\omega+\delta < h+\delta-1.
	\]
	In particular, $H_{\min}(\cC,\cC')$ is a bi-Lipschitz invariant, and so is $\omega=H_{\min}(\cC,\cC')-\delta$.
\end{corollary}

\begin{proof}
	Set $F(X,Y):=f(X+\gamma(Y),Y)$ and $G(X,Y):=F(X,Y)-F(0,Y)$, so that $G_X=F_X$. Since $F_X(0,Y)\equiv 0$, the series $G$ has no terms of $X$--degree $0$ or $1$.
	
	Let $\nu=[\delta,1]$ and let $\omega$ be the $Y$--intercept of the supporting line of $E_\delta$ (cf.\ Figure~\ref{fig:4}). Then there exists a non-monomial polynomial $P_\delta\in\C[Z]$ such that
	\[
	\In_\nu(F_X)(X,Y)=Y^\omega\,P_\delta\!\left(\frac{X}{Y^\delta}\right).
	\]
	Since $F_X(0,Y)\equiv 0$, the face polynomial satisfies $P_\delta(0)=0$. Set
	\[
	Q_\delta(Z):=\int_0^Z P_\delta(t)\,dt\in\C[Z],
	\]
	so that $Q'_\delta=P_\delta$ and, since $G_X=F_X$,
	\[
	\In_\nu(G)(X,Y)=Y^{\omega+\delta}\,Q_\delta\!\left(\frac{X}{Y^\delta}\right).
	\]
	As $P_\delta$ is not a monomial, $Q_\delta$ is not of the form $c(Z-a)^N$ (note that $Q_\delta(0)=0$), hence
	$Q'_\delta=P_\delta$ has a zero $p\notin Z(Q_\delta)$; in particular $p\in\C^*$ and $Q_\delta(p)\neq 0$.

	By the Newton--Puiseux algorithm applied to the edge $E_\delta$, the equality $P_\delta(p)=0$ yields a Puiseux solution $\eta(Y)=pY^\delta+\hot$ of $F_X(X,Y)=0$. Set $\gamma'(Y):=\gamma(Y)+\eta(Y)$ and $\cC':=\GC(\gamma'_*)$. We claim that $f(\gamma'(y),y)\not\equiv 0$, hence $\gamma'$ is a polar arc. Indeed, if $f(\gamma'(y),y)\equiv 0$, then $\gamma'$ is a Puiseux root of $f$ and therefore $\ord(\gamma-\gamma')\le \delta_{top}$ by the definition of $\delta_{top}$ (cf.\ Section~\ref{s:newtonpoly}), contradicting $\delta>\delta_{top}$.

	Since $\delta>\delta_{top}$, Lemma~\ref{lem:h-and-a-stable} implies $\ord f(\gamma'(y),y)=h$ and the leading coefficient at order $h$ coincides with that of $f(\gamma(y),y)$. Therefore
	\[
	H(\gamma,\gamma')=\ord\bigl(F(0,Y)-F(\eta(Y),Y)\bigr)=\ord\,G(\eta(Y),Y).
	\]
	Using $\In_\nu(G)$ and $Q_\delta(p)\neq 0$ we get $G(\eta(Y),Y)=Y^{\omega+\delta}Q_\delta(p)+\hot$, hence $H(\gamma,\gamma')=\omega+\delta$.
	
	By definition of $H_{\min}(\cC,\cC')$ as the minimum of $H(\alpha,\alpha')$ among polar arcs $\alpha_*\subset\cC$, $\alpha'_*\subset\cC'$ with $\ord(\alpha-\alpha')=\delta$, the same $\nu$--valuation estimate gives $H(\alpha,\alpha')\ge \omega+\delta$ for all such pairs. Thus $H_{\min}(\cC,\cC')=\omega+\delta$.
	
	Finally, let $(m^*,q^*)$ be a Newton dot of $F_X$ such that $(m^*+1,q^*)$ lies on the line $L$ of co-slope $d_{\cC}$ through $(1,h-1)$ (Section~\ref{s:newtonpoly}), so $q^*+d_{\cC}m^*=h-1$. Since $\delta<d_{\cC}$,
	\[
	\omega\le q^*+\delta m^*<q^*+d_{\cC}m^*=h-1,
	\]
	hence $H_{\min}(\cC,\cC')=\omega+\delta<h+\delta-1$.

	Since $H_{\min}(\cC,\cC')=\omega+\delta<h+\delta-1$, Theorem~\ref{c:invar1} applies to any pair realising $H_{\min}$, and yields that $H_{\min}(\cC,\cC')$ is preserved by bi-Lipschitz right-equivalence. Consequently, $\omega=H_{\min}-\delta$ is preserved as well (and $\delta$ is preserved as a contact below the canyon degree, cf.\ Remark~\ref{rem:puiseux-canyons} and as a contact order between the two canyons, cf. Subsection~\ref{ss:lipinvar}).
\end{proof}

\begin{figure}[h]
	\centering
	\begin{tikzpicture}[scale=0.8]
		\draw[->] (-0.5, 0) -- (6, 0) ;
		\draw[->] (0, -0.5) -- (0, 9.8) ;
		
		\foreach \x in {1,2,3,4,5} {
			\draw[gray, dashed] (\x, 0) -- (\x, 9.3);
		}
		\foreach \y in {1,2,3,4,5,6,7,8,9} {
			\draw[gray, dashed] (0, \y) -- (5.5, \y);
		}

		\draw[thick] (5, 0) -- (4, 1);
		\draw[thick] (4, 1)--(3,2.5);
		\draw[thick] (3, 2.5) --(2,4.5)node[pos=0.5, below left] {$E_{\delta}$};
		\draw[thick] (2, 4.5) --(1,7.25);
		\draw[thick, dashed] (2, 4.5) --(0,8.5);
		\fill[black] (1, 7.25) circle (3pt) node[right]{$(m^*,q^*)$};
		\fill[black] (2, 4.5) circle (3pt);
		\fill[black] (3, 2.5) circle (3pt);
		\fill[black] (4, 1) circle (3pt) node[right] {$(\hat{m}_{top}-1, \hat{q}_{top})$};
		\fill[black] (5, 0) circle (3pt);
		\fill[black] (0, 8.5) circle (3pt) node[left]{$(0,\omega)$};

	\end{tikzpicture}
	\caption{Extension of $E_\delta$ to the vertical axis in $\PP(f_x,\gamma)$.}
	\label{fig:4}
\end{figure}
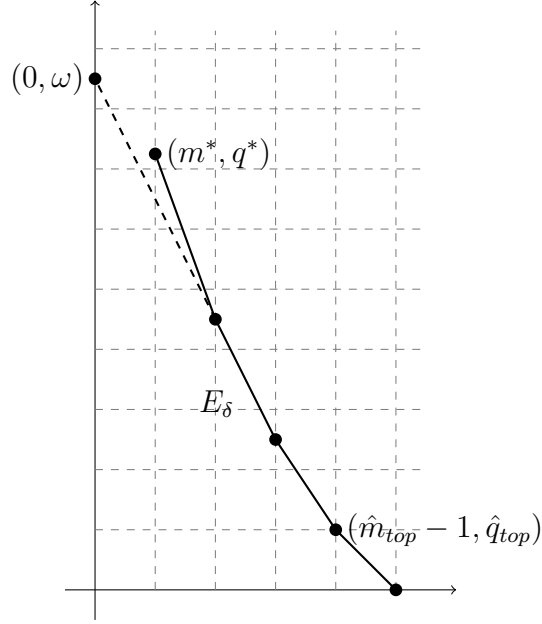

\medskip

\begin{theorem}\label{newton_inv}
	Let $f:(\C^2,0)\to(\C,0)$ be a mini-regular holomorphic function in $x$. Let $\cC=\GC(\gamma_*)$ be a gradient canyon of $f$ of degree $d_{\cC}>1$ and let $\gamma$ be any polar arc contained in $\cC$. Then the augmented Newton polygon $\widehat{\PP}(f_x,\cC)=\widehat{\PP}(f_x,\gamma)$ is a bi-Lipschitz invariant in the sense of Convention~\ref{conv:bilip}.
\end{theorem}

\begin{proof}
By Lemma~\ref{lem:augNP_canyon} the polygon $\widehat{\PP}(f_x,\gamma)$ depends only on the canyon, so we may fix an arbitrary polar arc $\gamma\in\cC$ and work with $\widehat{\PP}(f_x,\gamma)$. Set $h:=\ord f(\gamma(y),y)$ and $d:=d_{\cC}$.

Let $E_{top}$ be the top edge of $\PP(f,\gamma)$ and $\delta_{top}$ its co-slope, and let $E_{con}$ be the top edge of $\widehat{\PP}(f_x,\gamma)$ (its co-slope is $d$). We can group the compact edges of $\widehat{\PP}(f_x,\gamma)$ in the following way (all of them except $E_{con}$ are compact edges of $\PP(f_x,\gamma)$, since every dot of $F_X$ lies on or above the line of co-slope $d$ through $(0,h-1)$, cf.\ the construction of $L$ in Section~\ref{s:newtonpoly}):
\begin{itemize}
	\item[1.] the edges whose co-slopes are less than $\delta_{top}$.
	\item[2.] the top edge $E_{con}$, and the edges whose co-slopes are greater than $\delta_{top}$ and less than $d$. Denote the latter by $E_1, \ldots, E_s$, and their co-slopes by $\delta_1,\ldots ,\delta_s$. We can assume that $\delta_1>\ldots >\delta_s$.
	\item[3.] the edge whose co-slope is equal to $\delta_{top}$. We denote it by $E_h$.
\end{itemize}

\medskip\noindent
\textbf{Preliminary reduction.}
If $\PP(f_x,\gamma)$ has no compact edge whose co-slope lies in the open interval $(\delta_{top},d)$, i.e.\ if $s=0$, then group~2 consists of $E_{con}$ only and the augmented polygon $\widehat{\PP}(f_x,\gamma)$ is determined by the topological part from Case~1 together with the top edge $E_{con}$ of co-slope $d$ through the point $(0,h-1)$; the supporting line of $E_{con}$ is determined by $d$ and $h$, which are bi-Lipschitz invariants (Convention~\ref{conv:bilip}). In this situation the proof reduces to Cases~1 and~3. Hence, in Case~2 we may assume $s\ge 1$.

\medskip
 
We will show that edges of each type are bi-Lipschitz invariants. 

\medskip

\noindent \textbf{Case 1.} 
Consider the part of $\PP(f,\gamma)$ consisting of compact edges whose co-slope is $<\delta_{top}$. By Lemma~\ref{lem:NP-below-deltatop}, this collection of edges, together with their co-slopes and horizontal lengths, is a topological invariant of $f$. The compact edges of $\PP(f_x,\gamma)=\PP(F_X)$ of co-slope $<\delta_{top}$ are obtained from the corresponding edges of $\PP(f,\gamma)=\PP(F)$ by translation by $[-1,0]$, hence they are topological, and therefore also bi-Lipschitz invariants.

\medskip

\noindent\textbf{Case 2.}
Assume $s\ge 1$. For each compact edge $E_i$ of $\PP(f_x,\gamma)$ with co-slope $\delta_i\in(\delta_{top},d)$, let $\omega_i$ be the $Y$--intercept of the supporting line of $E_i$ (see Figure~\ref{fig:4}).

By Corollary~\ref{cor:Hmin_equals_omega_plus_delta}, for each $i$ there exists a polar arc $\gamma_i'$ with $\ord(\gamma-\gamma_i')=\delta_i$ such that, writing $\cC_i'$ for the corresponding geometric or Puiseux-level canyon as in Remark~\ref{rem:puiseux-canyons}, one has $H_{\min}(\cC,\cC_i')=\omega_i+\delta_i<h+\delta_i-1$. Hence $\omega_i$ is a bi-Lipschitz invariant; moreover $\delta_i$ is a bi-Lipschitz invariant as a contact below the canyon degree; when the two canyons are geometrically distinct this is exactly the contact invariance recalled in Subsection~\ref{ss:lipinvar}, and in the Puiseux-level case it follows from Remark~\ref{rem:puiseux-canyons}.

Hence, for each $i$ the supporting line of $E_i$ is the line of co-slope $\delta_i$ with $Y$--intercept $\omega_i$, while the supporting line of $E_{con}$ is the line of co-slope $d$ through the fixed point $(0,h-1)$. These supporting lines are determined by bi-Lipschitz invariant data, and $\widehat{\PP}(f_x,\gamma)$ is their lower convex hull. Therefore the edges $E_{con},E_1,\ldots,E_s$ are bi-Lipschitz invariants.

\noindent \textbf{Case 3.} The position of $E_h$ is uniquely determined by the edges from Cases~1 and~2 and by the value $\delta_{top}$.

By Cases~1--3, every supporting line of $\widehat{\PP}(f_x,\gamma)$ is determined by bi-Lipschitz invariant data, hence preserved by the correspondence $\Phi$ of Convention~\ref{conv:bilip} (in Case~1, since $\varphi$ is a homeomorphism and $\Phi$ preserves the bar $B(h)$). Therefore $\widehat{\PP}(f_x,\cC)=\widehat{\PP}\bigl(g_x,\Phi(\cC)\bigr)$.

\end{proof}

\section{Multiplicity of gradient canyons}\label{s:multiplicity}
	
Throughout this section $f:(\C^2,0)\to(\C,0)$ is mini-regular in $x$ and
$\cC=\GC(\gamma_*)$ is a gradient canyon of degree $d:=d_\cC>1$. We use the Puiseux
factorisation of the derivative
\[
f_x(x,y)=v(x,y)\prod_{i=1}^{m-1}\bigl(x-\gamma_i(y)\bigr),\qquad v\ \text{a unit},
\]
the $\gamma_i\in\F_1$ being the Puiseux roots of $f_x$ listed with multiplicity: the
Puiseux conjugates of one irreducible polar curve occur as distinct indices, and a root occurring
with algebraic multiplicity is repeated accordingly. We call each $\gamma_i$ a \emph{Puiseux
	representative} of a polar and $\gamma_{i,*}$ its \emph{geometric (image) arc}. Put
\[
I(\cC):=\bigl\{\,i\in\{1,\dots,m-1\}\mid \gamma_{i,*}\in\cC\,\bigr\}.
\]
Since canyon membership is defined through the geometric arcs $\gamma_{i,*}$, the set $I(\cC)$ is a
union of full conjugacy classes: if $\gamma_{i,*}\in\cC$ then $\gamma_{j,*}\in\cC$ for every
conjugate $\gamma_j$ of $\gamma_i$, since conjugate representatives share the same geometric arc. We
write $h_\cC:=\ord f(\alpha(y),y)$ for the common order of $f$ along the arcs $\alpha\in\cC$, which is
the same for every such arc by \cite[Proposition~3.7(c)]{PT}.

\begin{definition}[Polar multiplicity]\label{def:PT-mult-canyon}
	The \emph{polar multiplicity} of $\cC$ is the number of Puiseux representatives $\gamma_i$ of polars
	(the Puiseux roots of $f_x$, listed with multiplicity as above) whose geometric arcs $\gamma_{i,*}$ lie in $\cC$,
	\[
	\mult(\cC):=\# I(\cC).
	\]
	Since one Puiseux representative of an irreducible polar curve lies in $\cC$ iff all its
	conjugates do, $\mult(\cC)$ counts every Puiseux representative of every polar curve contained in
	$\cC$; it coincides with the canyon multiplicity of \cite[Definition~2.3]{PT}.
\end{definition}

The goal of this section is to prove that $\mult(\cC)$ is a bi-Lipschitz invariant. We factor it as a
product of two quantities of a very different nature: a bar multiplicity
$\mult^{\mathrm{bar}}(\cC)$, read off the augmented Newton polygon, and a number of trunks
$b_\cC$, the number of contact classes of the polar representatives, equivalently the
total number of Puiseux trunk branches occurring in the canyon, whose invariance rests on the canyon-disk correspondence of \cite{PT}.

\medskip
\noindent\textbf{Bar multiplicity.}
Fix once and for all a Puiseux representative $\gamma$ of a polar in $\cC$.

\begin{definition}[Bar multiplicity]\label{def:bar-mult-canyon}
	The bar multiplicity of $\cC$ is
	\begin{equation}\label{eq:bar-mult-def}
		\mult^{\mathrm{bar}}(\cC):=\#\bigl\{\,i\in I(\cC)\mid \ord\bigl(\gamma_i(y)-\gamma(y)\bigr)\ge d\,\bigr\}
	\end{equation}
	(the indices counted with the algebraic multiplicity of the roots). Equivalently, it is the number
	of Puiseux representatives in $\cC$ sharing with $\gamma$ the truncation $\gamma^{<d}$ below order
	$d$. That this number does not depend on the chosen $\gamma$ is shown right after
	Lemma~\ref{lem:mu_equals_length}.
\end{definition}

\noindent\emph{Relation with the Kuo--Lu bar.} The polar representatives in $I(\cC)$ have the same
canyon order $h_\cC$, and are therefore associated with the same $B(h_\cC)$-cluster of the Kuo--Lu
tree $T(f)$ in the sense recalled in Subsection~\ref{ss:kuolu-tree} (without committing to a single
literal bar, as the representatives may be distributed over conjugate departing components); then
\eqref{eq:bar-mult-def} counts those representatives that, in addition, still share the trunk
$\gamma^{<d}$ of $\gamma$ at the rational level $d$ (which need not be a bar height). We stress that
this refinement is not read off $T(f)$: it involves the genuine contacts
$\ord(\gamma_i-\gamma)$ between the polar roots, and these are not determined by \(T(f)\): as shown by
\cite[\S4, Example~4.4]{kuo-lu}, the polar tree \(T(f_x)\) (Subsection~\ref{ss:kuolu-tree})
cannot be reconstructed from \(T(f)\). The
quantity $\mult^{\mathrm{bar}}(\cC)$ is instead computed metrically, by the augmented Newton polygon,
as the next lemma shows. We do not name these groups, as only their common cardinality
$\mult^{\mathrm{bar}}(\cC)$ and their number $b_\cC$ (introduced below) are used.

The following lemma is the heart of the section: it identifies the bar multiplicity with a piece of
the augmented Newton polygon, hence with a bi-Lipschitz invariant.

\begin{lemma}\label{lem:mu_equals_length}
	With the notation above, let $E_{con}$ be the top compact edge of the augmented Newton polygon
	$\widehat{\PP}(f_x,\cC)$. Then
	\[
	\operatorname{len}_x(E_{con})=\mult^{\mathrm{bar}}(\cC).
	\]
\end{lemma}

\begin{proof}
	Assume $\gamma=\gamma_1$ and set $\Delta_i(y):=\gamma_i(y)-\gamma(y)$, so $\Delta_1\equiv 0$. To avoid
	a clash with the notation $F(X,Y)=f(X+\gamma(Y),Y)$ of Section~\ref{s:newtonpoly}, write
	\[
	\Psi(X,Y):=f_x(X+\gamma(Y),Y)=\tilde v(X,Y)\prod_{i=1}^{m-1}\bigl(X-\Delta_i(Y)\bigr),
	\]
	$\tilde v(X,Y):=v(X+\gamma(Y),Y)$ a unit, so that $\Psi=F_X$ and $\PP(\Psi)=\PP(F_X)=\PP(f_x,\gamma)$.
	
	Set $h:=\ord f(\gamma(y),y)=h_\cC$. By \eqref{eq:bar-mult-def} an index $i\in I(\cC)$ contributes to
	$\mult^{\mathrm{bar}}(\cC)$ iff $\ord\Delta_i\ge d$; conversely, if $\ord\Delta_i\ge d$ then
	$\gamma_{i,*}\in\GC(\gamma_*)=\cC$, so $i\in I(\cC)$. Hence
	\[
	\mult^{\mathrm{bar}}(\cC)=\#\bigl\{\, i\in\{1,\ldots,m-1\} \mid \ord\Delta_i\ge d \,\bigr\}.
	\]
	This counts the roots sharing the trunk of $\gamma$, not all of $\cC$: a conjugate
	representative $\gamma_i$ of $\gamma$ with $\ord\Delta_i<d$ has $\gamma_{i,*}\in\cC$ yet is excluded.
	
	Consider the weight $\nu=[d,1]$, i.e.\ the weighted order with $\nu(X)=d$ and $\nu(Y)=1$ (so that $\nu(X^aY^b)=ad+b$). Let $S_\nu$ be the set of Newton dots of $\Psi$ attaining the minimal $\nu$-value; this set may be a compact face or a single vertex of $\PP(\Psi)$. The corresponding initial form is $\In_\nu(\Psi)$. For each factor:
	\begin{itemize}
		\item if $\ord\Delta_i<d$, then $\In_\nu(X-\Delta_i)=-\In_\nu(\Delta_i)$ has $X$-degree $0$;
		\item if $\ord\Delta_i=d$, then $\In_\nu(X-\Delta_i)=X-\In_\nu(\Delta_i)$ has $X$-degree $1$;
		\item if $\ord\Delta_i>d$ (or $\Delta_i\equiv 0$), then $\In_\nu(X-\Delta_i)=X$ has $X$-degree $1$.
	\end{itemize}
	Since $\tilde v$ is a unit it does not affect initial forms, so
	\[
	\deg_X \In_\nu(\Psi)=\#\bigl\{\, i \mid \ord\Delta_i\ge d \,\bigr\}
	=\mult^{\mathrm{bar}}(\cC).
	\]
	
	It remains to identify $\deg_X\In_\nu(\Psi)$ with $\operatorname{len}_x(E_{con})$. Since $\gamma$ is a polar arc, $\Psi(0,Y)=f_x(\gamma(y),y)\equiv 0$, so $\PP(\Psi)$ has no dot on the line $X=0$ and every dot of $\Psi$ has $X$-degree $\ge1$. Moreover, by the definition of the gradient degree, equivalently by \cite[Lemma~3.4]{PT}, the line of co-slope $d$ through the added point $(0,h-1)$ supports $\PP(\Psi)$, and its contact set with $\PP(\Psi)$ is precisely $S_\nu$. Indeed,
	\[
	\frac{d}{dy}f(\gamma(y),y)=f_y(\gamma(y),y)
	\]
	because $f_x(\gamma(y),y)\equiv0$, and hence
	\[
	\ord f_y(\gamma(y),y)=\ord f(\gamma(y),y)-1=h-1.
	\]
	Thus the added point $(0,h-1)$ lies on this supporting line. After adjoining $(0,h-1)$, the same line gives the top compact edge $E_{con}$ of $\widehat{\PP}(f_x,\cC)$. Its left endpoint is $(0,h-1)$ and its right endpoint has $X$-coordinate equal to the maximal $X$-exponent occurring in $\In_\nu(\Psi)$, namely $\deg_X\In_\nu(\Psi)$. Therefore
	\[
	\operatorname{len}_x(E_{con})
	=
	\deg_X \In_\nu(\Psi)
	=
	\mult^{\mathrm{bar}}(\cC),
	\]
	as claimed.
\end{proof}

\noindent\emph{Independence of the choice of $\gamma$.} By Lemma~\ref{lem:augNP_canyon} the augmented
polygon $\widehat{\PP}(f_x,\cC)$ depends only on the canyon $\cC$, not on the representative used to
form it. Hence $\operatorname{len}_x(E_{con})$ is intrinsic to $\cC$, and by
Lemma~\ref{lem:mu_equals_length} so is $\mult^{\mathrm{bar}}(\cC)$: the right-hand side of
\eqref{eq:bar-mult-def} is the same for every representative $\gamma$ of every polar in $\cC$.

\begin{corollary}\label{cor:bar-mult}
	If $\cC$ is a gradient canyon of degree $d_\cC>1$, then $\mult^{\mathrm{bar}}(\cC)$ is a bi-Lipschitz
	invariant in the sense of Convention~\ref{conv:bilip}.
\end{corollary}
\begin{proof}
	By Theorem~\ref{newton_inv} the augmented polygon $\widehat{\PP}(f_x,\cC)$ is a bi-Lipschitz
	invariant, hence so is the horizontal length of its top edge $E_{con}$, which equals
	$\mult^{\mathrm{bar}}(\cC)$ by Lemma~\ref{lem:mu_equals_length}.
\end{proof}

\medskip
\noindent\textbf{The trunk and the number of trunks.}
For a Puiseux representative $\gamma_i(y)=\sum_q a_{iq}y^q$ write
\[
\gamma_i^{<d}(y):=\sum_{q<d}a_{iq}y^q
\]
for its \emph{Puiseux trunk}, the sum of the terms of order strictly below $d$.
If this sum is zero, we use the convention that it defines the smooth arc \(\{x=0\}\).
By the definition of the truncation, two representatives of $I(\cC)$ have contact
$\ge d$ iff they have the same Puiseux trunk.

Let
\begin{equation}\label{eq:equiv-d}
	i\equiv_d j
	\quad\Longleftrightarrow\quad
	\ord(\gamma_i-\gamma_j)\ge d .
\end{equation}

Thus the $\equiv_d$-classes are precisely the classes of representatives with
the same Puiseux trunk. For each such class choose one representative
$\gamma_i$ and take the geometric arc $(\gamma_i^{<d})_*$. We define the
\emph{trunk curve} of $\cC$ to be the reduced curve germ
\[
\operatorname{Tr}(\cC):=
\bigcup_{[i]\in I(\cC)/\!\equiv_d}(\gamma_i^{<d})_* ,
\]
where repeated components are taken only once. Thus $\operatorname{Tr}(\cC)$
need not be irreducible. Its irreducible components are the reduced geometric
arcs determined by the distinct trunk branches occurring in the canyon. Algebraic multiplicities of polar roots are not included in
\(\operatorname{Tr}(\cC)\); they are recorded in
\(\mult^{\mathrm{bar}}(\cC)\).

\begin{definition}[Number of trunks]\label{def:number-trunks}
	The \emph{number of trunks} of $\cC$ is
	\[
	b_\cC:=\#\bigl(I(\cC)/\!\equiv_d\bigr).
	\]
	Equivalently, \(b_\cC\) is the multiplicity at the origin of the reduced
	trunk curve:
	\[
	b_\cC=\mult_0\bigl(\operatorname{Tr}(\cC)\bigr).
	\]
	Indeed, each irreducible component of \(\operatorname{Tr}(\cC)\) is defined
	by one Puiseux trunk together with all its conjugates; the multiplicity of
	that component equals the number of distinct conjugate Puiseux trunks
	parametrising it. Summing over the irreducible components gives exactly the
	number of \(\equiv_d\)-classes. Repeated polar roots do not create new
	trunks; they are counted instead in \(\mult^{\mathrm{bar}}(\cC)\).
\end{definition}

\begin{lemma}[Multiplicity decomposition]\label{lem:bar-decomposition}
	$\mult(\cC)=b_\cC\cdot\mult^{\mathrm{bar}}(\cC)$.
\end{lemma}
\begin{proof}
	Group the indices of $I(\cC)$ according to their Puiseux trunk. Two indices fall in the same group
	iff their representatives have contact $\ge d$, so each group has cardinality
	$\mult^{\mathrm{bar}}(\cC)$ by the independence statement following
	Lemma~\ref{lem:mu_equals_length}. The number of groups is the number of distinct trunks occurring in
	$I(\cC)$, which equals $b_\cC$ by the discussion above. As the groups partition $I(\cC)$,
	$\mult(\cC)=\#I(\cC)=b_\cC\cdot\mult^{\mathrm{bar}}(\cC)$.
\end{proof}

We illustrate the two factors on the cusp, in two regimes; in each case the count is governed by the
trunk.

\begin{itemize}
	\item For $f=(x^2-y^3)^2+y^{6}$ the cusp $\{x^2=y^3\}$ is a single irreducible polar curve, with
	conjugate representatives $\pm y^{3/2}$; one computes $d=2$. The two trunks $(\pm y^{3/2})^{<2}=\pm
	y^{3/2}$ are conjugate and distinct, so $b_\cC=m_{puiseux}(y^{3/2})=2$, each group of bar
	multiplicity $1$, and $\operatorname{Tr}(\cC)=\{x^2=y^3\}$. (The third polar $\gamma_0=0$ has contact
	$\tfrac32<2$ with $\pm y^{3/2}$ and a different geometric arc, hence lies in a different canyon.)
	\item For $f=(x^2-y^3)^2+y^{5}$ one computes $d=\tfrac43$, and
	$\ord\bigl(y^{3/2}-(-y^{3/2})\bigr)=\tfrac32\ge\tfrac43$, so $\pm y^{3/2}$ now share a trunk, which
	also contains $\gamma_0=0$ (contact $\tfrac32\ge d$). The common trunk is $\gamma^{<4/3}=0$, fixed by
	conjugation, so $b_\cC=1$: a single group of bar multiplicity $3$, with
	$\operatorname{Tr}(\cC)=\{x=0\}$. Here the value $3$ counts the three Puiseux representatives
	$0,\pm y^{3/2}$, which form two distinct geometric polar curves $\{x=0\}$ and $\{x^2=y^3\}$;
	it is not a count of geometric polar curves.
\end{itemize}

\begin{lemma}[The number of trunks is a bi-Lipschitz invariant]\label{lem:number-canyon-bars-invariant}
	Let $\cC$ be a gradient canyon of degree $d>1$, with trunk curve $\operatorname{Tr}(\cC)$. Then
	\begin{enumerate}
		\item[(i)] the intersection multiplicity of $\{f=0\}$ with the trunk curve is
		\[
		\mult_0\bigl(\{f=0\},\operatorname{Tr}(\cC)\bigr)=b_\cC\,h_\cC;
		\]
		\item[(ii)] $b_\cC$ is a bi-Lipschitz invariant in the sense of Convention~\ref{conv:bilip}.
	\end{enumerate}
\end{lemma}

\begin{proof}
	\emph{(i).} Write
	\[
	\operatorname{Tr}(\cC)=D_1\cup\cdots\cup D_s
	\]
	for the decomposition into irreducible reduced components. For each component
	\(D_\ell\), choose a Puiseux trunk \(\tau_\ell\) parametrising it, and let
	\(N_\ell=\mult_0(D_\ell)\). Then \(D_\ell\) has a normalisation of the form
	\[
	t\mapsto(\tau_\ell(t^{N_\ell}),t^{N_\ell}).
	\]
	Moreover, \(\tau_\ell\) differs from some polar representative in \(I(\cC)\)
	by terms of order at least \(d\); hence the geometric arc
	\((\tau_\ell)_*=D_\ell\) belongs to the canyon \(\cC\). Since the canyon order
	\(h_\cC\) is the same for every arc of \(\cC\)
	\cite[Proposition~3.7(c)]{PT}, we have
	\[
	f(\tau_\ell(y),y)=a_\ell y^{h_\cC}+\hot,\qquad a_\ell\neq0.
	\]
	Consequently,
	\[
	\mult_0(\{f=0\},D_\ell)
	=
	\ord_t f(\tau_\ell(t^{N_\ell}),t^{N_\ell})
	=
	N_\ell h_\cC.
	\]
	Summing over the irreducible components of the reduced trunk curve gives
	\[
	\mult_0\bigl(\{f=0\},\operatorname{Tr}(\cC)\bigr)
	=
	\sum_{\ell=1}^s \mult_0(\{f=0\},D_\ell)
	=
	h_\cC\sum_{\ell=1}^s N_\ell.
	\]
	By Definition~\ref{def:number-trunks},
	\[
	\sum_{\ell=1}^s N_\ell=\mult_0(\operatorname{Tr}(\cC))=b_\cC.
	\]
	Thus
	\[
	\mult_0\bigl(\{f=0\},\operatorname{Tr}(\cC)\bigr)=b_\cC h_\cC.
	\]
	
	\emph{(ii).} Fix the canyon correspondence $\Phi$ of Convention~\ref{conv:bilip}, and let
	$N_\cC$ be the number of connected canyon disks associated with the reduced trunk curve
	$\operatorname{Tr}(\cC)$ in a Milnor fibre; equivalently, take the union of the horn pieces centred
	at the distinct Puiseux trunks, with exponent $e<d$ sufficiently close to $d$.
	By \cite[(24)--(25) and the statement following~(25)]{PT}, these disks are counted by the
	intersection with the corresponding geometric truncation; in the present notation,
	\[
	N_\cC=\mult_0(\{f=0\},\operatorname{Tr}(\cC)).
	\]
	By \cite[Theorem~5.8]{PT}, applied canyon by canyon, the bi-Lipschitz map gives a
	degree-preserving bijection between the canyon disks of $\cC$ and those of $\Phi(\cC)$. Therefore
	$N_\cC=N_{\Phi(\cC)}$, or equivalently
	\[
	\mult_0(\{f=0\},\operatorname{Tr}(\cC))
	=
	\mult_0(\{g=0\},\operatorname{Tr}(\Phi(\cC))).
	\]
	Writing (i) on both sides gives $b_\cC\,h_\cC=b_{\Phi(\cC)}\,h_{\Phi(\cC)}$. Finally $h_\cC$ is itself a bi-Lipschitz
	invariant, so $h_\cC=h_{\Phi(\cC)}$. Dividing,
	$b_\cC=b_{\Phi(\cC)}$.
\end{proof}

\begin{corollary}\label{cor:polar-mult}
	If $\cC$ is a gradient canyon of degree $d_\cC>1$, then its polar multiplicity $\mult(\cC)$ is a
	bi-Lipschitz invariant in the sense of Convention~\ref{conv:bilip}.
\end{corollary}
\begin{proof}
	By Lemma~\ref{lem:bar-decomposition}, $\mult(\cC)=b_\cC\cdot\mult^{\mathrm{bar}}(\cC)$; the factor
	$\mult^{\mathrm{bar}}(\cC)$ is invariant by Corollary~\ref{cor:bar-mult} and the factor $b_\cC$ by
	Lemma~\ref{lem:number-canyon-bars-invariant}(ii). Hence $\mult(\cC)=\mult\bigl(\Phi(\cC)\bigr)$.
\end{proof}

\begin{remark}\label{rem:bars-symmetry}
	Two points deserve emphasis, both visible on the cusp examples. First, the bar multiplicity is
	governed by the Puiseux contact $\ord(\gamma_i-\gamma)$ between fixed representatives, whereas
	canyon membership is governed by the geometric arcs; consequently two conjugate
	representatives of one irreducible polar curve always lie in the same canyon, but share a trunk (and
	are counted together by $\mult^{\mathrm{bar}}$) only when their contact is $\ge d$, splitting into
	distinct trunks otherwise. Second, $\mult^{\mathrm{bar}}(\cC)$ and $\mult(\cC)$ count Puiseux
	representatives (with algebraic multiplicity), not geometric polar curves: as the second cusp
	example shows, bar multiplicity $3$ may be carried by two distinct geometric curves. None of the
	arguments above requires the polar locus to be irreducible or the polar roots simple.
\end{remark}

\section{The total curvature of a gradient canyon}\label{s:curvature}

The bi-Lipschitz invariance of the canyon multiplicity has a geometric counterpart: the total
Gaussian curvature that the level curves of $f$ concentrate along a gradient canyon is itself a
bi-Lipschitz invariant.

Let $f$ be mini-regular in $x$. At a regular point of a level curve $S_c:=\{f=c\}$,
$0<|c|\ll1$, this curve is a smooth real surface in $\C^2$ whose Gaussian curvature, in the
sign convention of \cite{KKP} for which the total curvature below is non-negative, is given by
\cite[(1.1)]{KKP}:
\[
K=\frac{2\,|\Delta_f|^2}{\bigl(|f_x|^2+|f_y|^2\bigr)^3},\qquad
\Delta_f:=\det\begin{pmatrix} f_{xx}&f_{xy}&f_x\\ f_{yx}&f_{yy}&f_y\\ f_x&f_y&0\end{pmatrix}.
\]
For $1\le e<\infty$, $0<\rho<\infty$ and $0<\eta\ll1$, we use the horn domain of
\cite[Definition~2.9]{KKP}, written here in the coordinates $(x,y)$:
\[
\Horn^{(e)}(\alpha_*;\rho;\eta):=
\bigl\{(x,y)\in\operatorname{Disc}(0;\eta)\mid
x=J^{(e)}(\alpha)(y)+c y^e,\ |c|\le\rho\bigr\}.
\]
Here $J^{(e)}(\alpha)$ denotes the $e$-jet of the Puiseux arc $\alpha$, i.e.\ the sum of the terms of $\alpha(y)$ of order $\le e$.
For a polar arc $\gamma$ of degree $d=d_\cC>1$ and $\cC=\GC(\gamma_*)=\mathcal{L}^{(d)}(\gamma_*)$,
the curvature concentrated along $\cC$ is the \emph{total asymptotic Gaussian curvature}
\[
\mathcal{M}_f(\cC):=\lim_{\rho\to\infty}\ \lim_{\eta\to0}\
\left(\lim_{c\to0}\int_{S_c\cap\,\Horn^{(d)}(\gamma_*;\rho;\eta)} K\,dS\right).
\]
This is exactly $\mathcal{M}_f(\mathcal{L}^{(d)}(\gamma_*))$ in the notation of \cite[(2.19)--(2.20)]{KKP}; the
order of the two inner limits is the one specified in \cite[Note following~(2.20)]{KKP}.

Recall that the Milnor number of $f$ equals $\sum_{i=1}^{m-1}\bigl(\ord f(\gamma_i(y),y)-1\bigr)$, see e.g.\ \cite{PT,KKP}; the \emph{partial Milnor number} of $\cC$ is the part of this sum carried by $\cC$, namely $\mu_f(\cC):=\sum_{i\in I(\cC)}\bigl(\ord f(\gamma_i(y),y)-1\bigr)$.

\begin{theorem}[Koike--Kuo--P\u aunescu {\cite[Theorem~C]{KKP}}]
	\label{thm:KKP-C}
	For a gradient canyon $\cC=\GC(\gamma_*)$ with $1<d_\cC<\infty$,
	\[
	\mathcal{M}_f(\cC)=2\pi\bigl[\,\mu_f(\cC)+\mult(\cC)\,\bigr],
	\]
	where $\mu_f(\cC)$ is the partial Milnor number of the canyon and $\mult(\cC)$ its polar
	multiplicity (Definition~\ref{def:PT-mult-canyon}).
\end{theorem}

Since all polar roots contained in $\cC$ share the order $h_\cC=\ord f(\gamma_i(y),y)$
(Subsection~\ref{ss:order-h}), one has $\mu_f(\cC)=\mult(\cC)\,(h_\cC-1)$, and
Theorem~\ref{thm:KKP-C} simplifies to
\begin{equation}\label{eq:curv-mult}
	\mathcal{M}_f(\cC)=2\pi\,\mult(\cC)\,h_\cC .
\end{equation}

\begin{corollary}\label{cor:curvature}
	Let $f:(\C^2,0)\to(\C,0)$ be mini-regular in $x$, and let $\cC=\GC(\gamma_*)$ be a gradient
	canyon of degree $1<d_\cC<\infty$. Then $\mathcal{M}_f(\cC)$ is a bi-Lipschitz invariant:
	$\mathcal{M}_f(\cC)=\mathcal{M}_g(\Phi(\cC))$ in the notation of Convention~\ref{conv:bilip}.
\end{corollary}
\begin{proof}
	By \eqref{eq:curv-mult} it suffices that both factors be preserved by $\Phi$. The polar
	multiplicity is preserved by Corollary~\ref{cor:polar-mult}, and the order $h_\cC$ is preserved
	by the identity-card part of Convention~\ref{conv:bilip}. Hence
	\[
	\mathcal{M}_f(\cC)=2\pi\,\mult(\cC)\,h_\cC=2\pi\,\mult(\Phi(\cC))\,h_{\Phi(\cC)}
	=\mathcal{M}_g(\Phi(\cC)).\qedhere
	\]
\end{proof}

\begin{remark}
	This is an equality of numbers: a bi-Lipschitz homeomorphism need not preserve Gaussian
	curvature pointwise, yet the total curvature it concentrates in corresponding canyons agrees.
	Thus the curvature-concentration data of \cite[Theorem~C]{KKP}, a priori analytic-geometric, is
	determined by the bi-Lipschitz class of $f$.
\end{remark}

\section{Examples}\label{s:examples}

		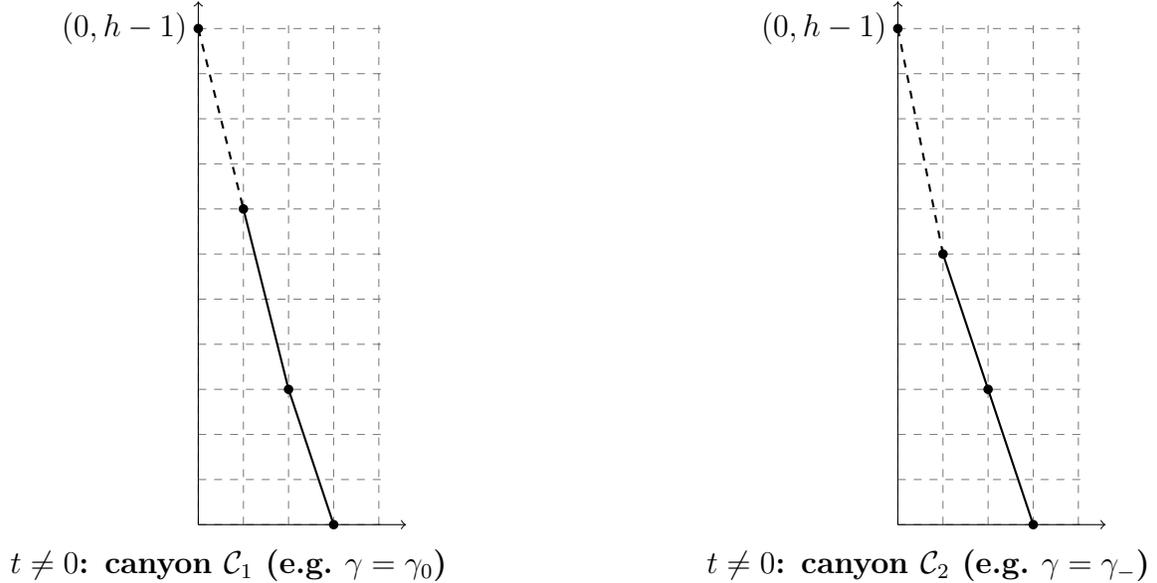
\begin{figure}[h]
	\centering
	
	\begin{minipage}{0.4\textwidth}
		\centering
		\begin{tikzpicture}[scale=0.6]
			\draw[step=1, dashed, color=gray] (0,0) grid (4.2,11.2);
			\draw[->] (0,0) -- (4.6,0);
			\draw[->] (0,0) -- (0,11.6);

			\fill[black] (3,0) circle (3pt);
			\fill[black] (2,3) circle (3pt);
			\fill[black] (1,7) circle (3pt);
			\draw[thick] (3,0) -- (2,3) -- (1,7);
			
			\fill[black] (0,11) circle (3pt) node[left] {$(0,h-1)$};
			\draw[thick, dashed] (0,11) -- (1,7);
		\end{tikzpicture}
		
		\smallskip
		\textbf{$t\neq 0$: canyon $\cC_1$ (e.g.\ $\gamma=\gamma_0$)}
	\end{minipage}
	\hfill
	\begin{minipage}{0.4\textwidth}
		\centering
		\begin{tikzpicture}[scale=0.6]
			\draw[step=1, dashed, color=gray] (0,0) grid (4.2,11.2);
			\draw[->] (0,0) -- (4.6,0);
			\draw[->] (0,0) -- (0,11.6);
			
			\fill[black] (3,0) circle (3pt);
			\fill[black] (2,3) circle (3pt);
			\fill[black] (1,6) circle (3pt);
			\draw[thick] (3,0) -- (2,3) -- (1,6);
			
			\fill[black] (0,11) circle (3pt) node[left] {$(0,h-1)$};
			\draw[thick, dashed] (0,11) -- (1,6);
		\end{tikzpicture}
		
		\smallskip
		\textbf{$t\neq 0$: canyon $\cC_2$ (e.g.\ $\gamma=\gamma_-$)}
	\end{minipage}
	
	\caption{For $t\neq 0$: the solid line segments are $\PP((f_t)_x,\gamma)$ for a representative polar arc $\gamma$ in the canyon, and the dashed segment is the extra edge in $\widehat{\PP}((f_t)_x,\cC)$ coming from the convex hull with $(0,h-1)$.}
	\label{fig:ft-canyon-tnz}
\end{figure}

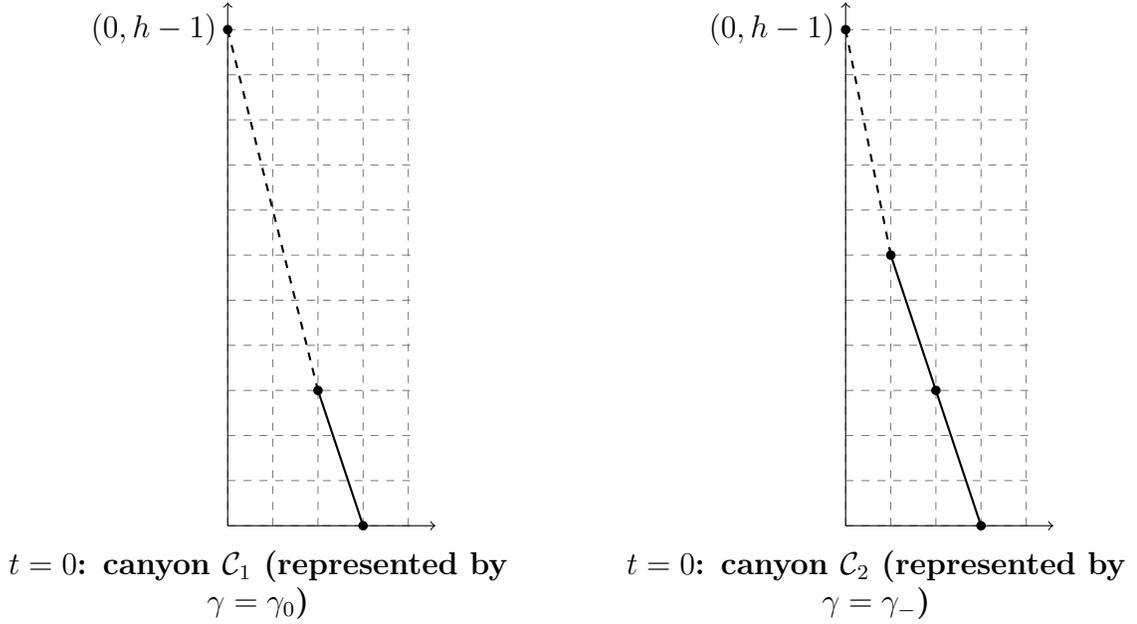
\begin{figure}[h]
	\centering
	
	\begin{minipage}{0.47\textwidth}
		\centering
		\begin{tikzpicture}[scale=0.6]
			\draw[step=1, dashed, color=gray] (0,0) grid (4.2,11.2);
			\draw[->] (0,0) -- (4.6,0);
			\draw[->] (0,0) -- (0,11.6);
			
			\fill[black] (3,0) circle (3pt);
			\fill[black] (2,3) circle (3pt);
			\draw[thick] (3,0) -- (2,3);
			
			\fill[black] (0,11) circle (3pt) node[left] {$(0,h-1)$};
			\draw[thick, dashed] (0,11) -- (2,3);
		\end{tikzpicture}
		
		\smallskip
		\textbf{$t=0$: canyon $\cC_1$ (represented by $\gamma=\gamma_0$)}
	\end{minipage}
	\hfill
	\begin{minipage}{0.47\textwidth}
		\centering
		\begin{tikzpicture}[scale=0.6]
			\draw[step=1, dashed, color=gray] (0,0) grid (4.2,11.2);
			\draw[->] (0,0) -- (4.6,0);
			\draw[->] (0,0) -- (0,11.6);
			
			\fill[black] (3,0) circle (3pt);
			\fill[black] (2,3) circle (3pt);
			\fill[black] (1,6) circle (3pt);
			\draw[thick] (3,0) -- (2,3) -- (1,6);
			
			\fill[black] (0,11) circle (3pt) node[left] {$(0,h-1)$};
			\draw[thick, dashed] (0,11) -- (1,6);
		\end{tikzpicture}
		
		\smallskip
		\textbf{$t=0$: canyon $\cC_2$ (represented by $\gamma=\gamma_-$)}
	\end{minipage}
	
	\caption{For $t=0$: the same convention as in Figure~\ref{fig:ft-canyon-tnz}.}
	\label{fig:ft-canyon-t0}
\end{figure}

\begin{example}\label{ex:newt}
	Consider the family
	\[
	f_t(x,y)=x^4+x^3y^3+y^{12}+t x^2y^7,\qquad t\in(\C,0).
	\]
	Then
	\[
	(f_t)_x=x(4x^2+3xy^3+2ty^7),
	\]
	hence for $t\neq 0$ it has three polar arcs tangent to $\ell=\{x=0\}$:
	\[
	\gamma_0(y)=0,\qquad 
	\gamma_+(y)=-\frac23\,t\,y^4+O(y^5),\qquad
	\gamma_-(y)=-\frac34\,y^3+\frac23\,t\,y^4+O(y^5).
	\]
	(For $t=0$ the root $\gamma_0$ becomes double.)
	
	\smallskip
	A direct computation shows that $\ord f_t(\gamma(y),y)=12$ for these polars; set $h=12$. Moreover, $\gamma_0$ and $\gamma_+$ belong to one gradient canyon $\cC_1$ of degree $d_{\cC_1}=4$, whereas $\gamma_-$ generates a different canyon $\cC_2$ of degree $d_{\cC_2}=5$, and their mutual contact is
	\[
	\delta=\ord(\gamma_- -\gamma_0)=3.
	\]

	\smallskip
	\noindent \emph{Constancy of the augmented Newton polygons.}
	For $\cC_1$ one gets
	\[
	\widehat{\PP}((f_t)_x,\cC_1)=\operatorname{Conv}\{(0,11),(2,3),(3,0)\},
	\]
	and for $\cC_2$ one gets
	\[
	\widehat{\PP}((f_t)_x,\cC_2)=\operatorname{Conv}\{(0,11),(1,6),(3,0)\}.
	\]
	In particular, both augmented polygons are independent of $t$ (even though $\PP((f_t)_x,\gamma)$ may acquire extra dots on these edges when $t\neq 0$).

	By Lemma~\ref{lem:augNP_canyon}, the augmented polygon is constant along each canyon, hence Figures~\ref{fig:ft-canyon-tnz} and~\ref{fig:ft-canyon-t0} depict $\widehat{\PP}((f_t)_x,\cC_1)$ and $\widehat{\PP}((f_t)_x,\cC_2)$ (solid line segments: $\PP((f_t)_x,\gamma)$ for a representative polar $\gamma\in\cC_i$; dashed segment: the additional edge coming from $(0,h-1)$).

	\smallskip
	\noindent\emph{Nevertheless the family is not bi-Lipschitz trivial.}
	Take $\gamma_0\in\cC_1$ and $\gamma_-\in\cC_2$. One computes
	\[
	f_t(\gamma_0(y),y)=y^{12},\qquad
	f_t(\gamma_-(y),y)=\frac{229}{256}y^{12}+\frac{9}{16}t\,y^{13}+O(y^{14}).
	\]
	After normalisation by the leading coefficients, the second-level invariant of Theorem~\ref{c:invar1} yields $H=13$ and
	\[
	\Delta_t:=\tilde a(\gamma_0)-\tilde a(\gamma_-)=-\frac{144}{229}\,t,
	\qquad\text{with }H<h+\delta-1=14.
	\]	
	Therefore, if $f_t$ and $f_s$ are bi-Lipschitz right-equivalent, then $\Delta_s=c^{\,h-H}\Delta_t=c^{-1}\Delta_t$ for some constant $c\neq 0$. On the other hand, applying the first-level HP-type invariant to the polar $\gamma_0$ (for which the leading coefficient is $a=1$) forces $c^{12}=1$, hence $s=c^{-1}t$ implies $t^{12}=s^{12}$. In particular, the family is not bi-Lipschitz trivial.
\end{example}

\begin{example}\label{ex:newt1}
	Consider the family
	\[
	f_t(x,y)=x^4+y^{12}+t x^3y^3+(1-t)x^2y^7,\qquad t\in(\C,0).
	\]
	Then
	\[
	(f_t)_x=x\bigl(4x^2+3txy^3+2(1-t)y^7\bigr).
	\]
	
	\smallskip
	\noindent\emph{Generic fibre.}
	Assume $t\neq 0,1$. Then $(f_t)_x$ has three polar arcs tangent to $\ell=\{x=0\}$:
	\[
	\gamma_0(y)=0,\qquad
	\gamma_1(y)=-\frac34\,t\,y^3+O(y^4),\qquad
	\gamma_2(y)=-\frac{2(1-t)}{3t}\,y^4+O(y^5).
	\]
	A direct computation gives $\ord f_t(\gamma_i(y),y)=12$ for $i=0,1,2$; set $h=12$. Moreover $\gamma_0$ and $\gamma_2$ lie in one gradient canyon $\cC_1$ of degree $d_{\cC_1}=4$, while $\gamma_1$ generates a different canyon $\cC_2$ of degree $d_{\cC_2}=5$, and their mutual contact is
	\[
	\delta=\ord(\gamma_1-\gamma_0)=3.
	\]	
	The corresponding augmented Newton polygons are (solid line: $\PP((f_t)_x,\gamma)$, dashed line: the extra edge from $(0,h-1)$):
	\[
	\widehat{\PP}((f_t)_x,\cC_1)=\operatorname{Conv}\{(0,11),(1,7),(2,3),(3,0)\},
	\]
	\[
	\widehat{\PP}((f_t)_x,\cC_2)=\operatorname{Conv}\{(0,11),(1,6),(3,0)\},
	\]
	see Figure~\ref{fig:newfam-generic}.

	\smallskip
	\noindent\emph{Special fibre $t=0$.}
	We have
	\[
	f_0(x,y)=x^4+x^2y^7+y^{12},\qquad (f_0)_x=2x(2x^2+y^7).
	\]
	Its polar arcs tangent to $\ell$ are
	\[
	\gamma_0(y)=0,\qquad \gamma_{\pm}(y)=\pm \xi\,y^{7/2}+O(y^{9/2}),
	\qquad \xi^2=-\tfrac12.
	\]
	Moreover, $f_0$ has one canyon generated by $\gamma_0$ and one geometric canyon represented by the conjugate pair $\gamma_\pm$, all of degree $4$. At the Puiseux level, $\gamma_+$ and $\gamma_-$ are separated, since their contact $\ord(\gamma_+-\gamma_-)=\tfrac72$ is below the common canyon degree $4$. For each of these Puiseux representatives one gets the same augmented polygon
	\[
	\widehat{\PP}((f_0)_x,\cC)=\operatorname{Conv}\{(0,11),(1,7),(3,0)\},
	\]
	see Figure~\ref{fig:newfam-t0}. In particular, the collection of augmented Newton polygons (and the canyon degrees they encode) differs for $t=0$ and for $t\neq 0,1$, hence the family $f_t$ is not bi-Lipschitz right-trivial.
	\end{example}

	\begin{figure}[h]
	\centering
	\begin{minipage}{0.47\textwidth}
		\centering
		\begin{tikzpicture}[scale=0.6]
			\draw[step=1, dashed, color=gray] (0,0) grid (4.2,11.2);
			\draw[->] (0,0) -- (4.6,0);
			\draw[->] (0,0) -- (0,11.6);
			\fill[black] (3,0) circle (3pt);
			\fill[black] (2,3) circle (3pt);
			\fill[black] (1,7) circle (3pt);
			\draw[thick] (3,0) -- (2,3) -- (1,7);
			\fill[black] (0,11) circle (3pt) node[left] {$(0,h-1)$};
			\draw[thick, dashed] (0,11) -- (1,7);
		\end{tikzpicture}
		\smallskip
		\textbf{$t\neq 0,1$: canyon $\cC_1$ (e.g.\ $\gamma=\gamma_0$)}
	\end{minipage}
	\hfill
	\begin{minipage}{0.47\textwidth}
		\centering
		\begin{tikzpicture}[scale=0.6]
			\draw[step=1, dashed, color=gray] (0,0) grid (4.2,11.2);
			\draw[->] (0,0) -- (4.6,0);
			\draw[->] (0,0) -- (0,11.6);
			\fill[black] (3,0) circle (3pt);
			\fill[black] (2,3) circle (3pt);
			\fill[black] (1,6) circle (3pt);
			\draw[thick] (3,0) -- (2,3) -- (1,6);
			\fill[black] (0,11) circle (3pt) node[left] {$(0,h-1)$};
			\draw[thick, dashed] (0,11) -- (1,6);
		\end{tikzpicture}
		\smallskip
		\textbf{$t\neq 0,1$: canyon $\cC_2$ (e.g.\ $\gamma=\gamma_1$)}
	\end{minipage}
	\caption{For $t\neq 0,1$: solid line is $\PP((f_t)_x,\gamma)$ for a representative polar arc $\gamma$ in the canyon, and dashed line is the extra edge in $\widehat{\PP}((f_t)_x,\cC)$ coming from the convex hull with $(0,h-1)$.}
	\label{fig:newfam-generic}
\end{figure}
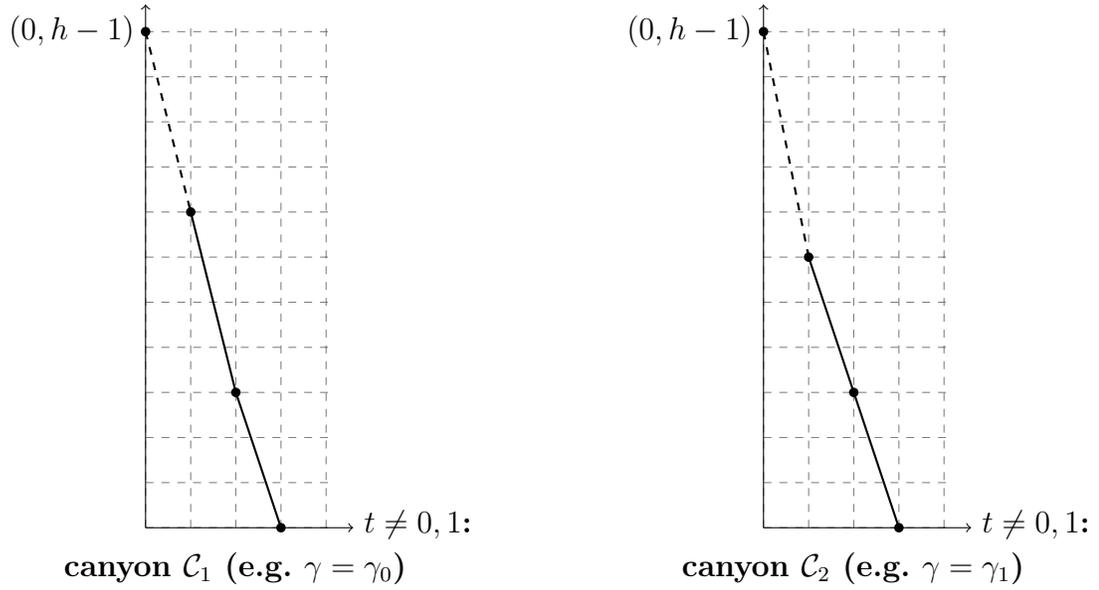

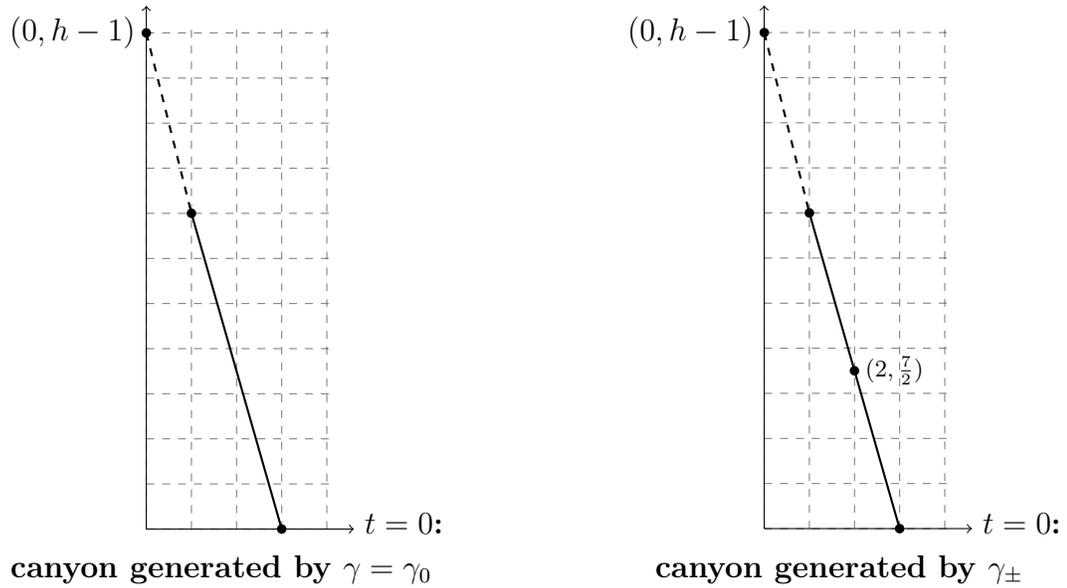
\begin{figure}[H]
	\centering
	\begin{minipage}{0.47\textwidth}
		\centering
		\begin{tikzpicture}[scale=0.6]
			\draw[step=1, dashed, color=gray] (0,0) grid (4.2,11.2);
			\draw[->] (0,0) -- (4.6,0);
			\draw[->] (0,0) -- (0,11.6);
			\fill[black] (3,0) circle (3pt);
			\fill[black] (1,7) circle (3pt);
			\draw[thick] (3,0) -- (1,7);
			\fill[black] (0,11) circle (3pt) node[left] {$(0,h-1)$};
			\draw[thick, dashed] (0,11) -- (1,7);
		\end{tikzpicture}
		\smallskip
		\textbf{$t=0$: canyon generated by $\gamma=\gamma_0$}
	\end{minipage}
	\hfill
	\begin{minipage}{0.47\textwidth}
		\centering
		\begin{tikzpicture}[scale=0.6]
			\draw[step=1, dashed, color=gray] (0,0) grid (4.2,11.2);
			\draw[->] (0,0) -- (4.6,0);
			\draw[->] (0,0) -- (0,11.6);
			\fill[black] (3,0) circle (3pt);
			\fill[black] (1,7) circle (3pt);
			\fill[black] (2,3.5) circle (3pt) node[right] {\scriptsize $(2,\tfrac{7}{2})$};
			\draw[thick] (3,0) -- (1,7);
			\fill[black] (0,11) circle (3pt) node[left] {$(0,h-1)$};
			\draw[thick, dashed] (0,11) -- (1,7);
		\end{tikzpicture}
		\smallskip
		\textbf{$t=0$: the two Puiseux-level canyons of $\gamma_+,\gamma_-$, same polygon}
	\end{minipage}
	\caption{For $t=0$: solid segments represent $\PP((f_0)_x,\gamma)$ for a representative polar root $\gamma$ in the indicated canyon, and the dashed segment is the extra edge in $\widehat{\PP}((f_0)_x,\cC)$ coming from the convex hull with $(0,h-1)$. Although conjugate, $\gamma_+,\gamma_-$ are separated at the Puiseux level since their contact $\tfrac72$ is below the degree $4$; the two representatives yield identical augmented polygons.}
	\label{fig:newfam-t0}
\end{figure}

\section*{Glossary of notation}

For the reader's convenience, we list the main notation, together with the place where it is introduced.

\begin{description}[font=\normalfont, leftmargin=1.5em, itemsep=2pt]
	\item[{$\F$, $\F_1$; $\ord\eta$, $m_{puiseux}(\eta)$, $\eta_*$}] convergent fractional power series in $y$, and those of order $\ge 1$; order, Puiseux multiplicity and curve germ of a Puiseux arc $\eta$ (Section~\ref{s:prelim}).
	\item[{$h(\gamma)$}] the order $\ord f(\gamma(y),y)$ of $f$ along a polar arc $\gamma$ (Definition~\ref{def:polar-arc}).
	\item[{$d(\gamma)$, $\GC(\gamma_*)$, $d_\cC$}] gradient degree of $\gamma$, gradient canyon of $\gamma_*$, degree of a canyon $\cC$ (Subsections~\ref{ss:degree} and~\ref{ss:lipinvar}).
	\item[{$T(f)$, $T^*(f_x)$, $T(f_x)$; $B(h)$}] Kuo--Lu tree-model of $f$, auxiliary tree, tree-model of $f_x$; bar from which the polar arcs $\gamma$ with $h(\gamma)=h$ depart (Subsection~\ref{ss:kuolu-tree}).
	\item[{$h_\cC$, $a_{h_\cC}$}] common value of $h(\gamma)$ for the polar arcs $\gamma$ in $\cC$, and the coefficient of $y^{h_\cC}$ in $f(\gamma(y),y)$ (Subsections~\ref{ss:order-h} and~\ref{ss:lipinvar}).
	\item[{$G_\ell(f)$, $G_{\ell,d,B(h)}(f)$, $G_{\ell,d,B(h),\omega}(f)$}] canyons tangent to $\ell$; cluster; sub-cluster, where $\omega=K_{\ell,d,B(h),i}(f)$ (Subsections~\ref{ss:clusters} and~\ref{ss:subclusters}).
	\item[{$k(i,j)$, $K_{\ell,d,B(h),i}(f)$}] order of contact of two canyons of a cluster; contact orders of $\GC_i(f)$, counted with multiplicity (Subsections~\ref{ss:contact} and~\ref{ss:subclusters}).
	\item[{$c$, $\Phi$}] constant of item~(d) (Subsection~\ref{ss:lipinvar}); bijection between the canyons of degree $>1$ of $f$ and of $g$ (Convention~\ref{conv:bilip}).
	\item[{$\delta$, $H$, $\tilde a$, $\tilde a'$}] contact order, second-level exponent and coefficients (Subsection~\ref{ss:second-level}).
	\item[{$\PP(f)$, $\PP(f,\eta)$, $F$}] Newton polygon of $f$; Newton polygon of $f$ relative to $\eta$, i.e.\ of $F(X,Y)=f(X+\eta(Y),Y)$ (Section~\ref{s:newtonpoly}).
	\item[{co-slope, $\nu=[a,b]$, $\ord_\nu$, $\In_\nu$, $\Supp$}] co-slope of an edge; weight, weighted order, initial form; support (Section~\ref{s:newtonpoly}).
	\item[{$E_{top}$, $\delta_{top}$; $L$, $(m^*,q^*)$}] top edge of $\PP(f,\gamma)$ and its co-slope; line of co-slope $d(\gamma)$ through $(1,h-1)$, and a dot of $F_X$ with $(m^*+1,q^*)\in L$ (Section~\ref{s:newtonpoly}).
	\item[{$\widehat{\PP}(f_x,\gamma)$, $\widehat{\PP}(f_x,\cC)$}] augmented Newton polygon (Definition~\ref{def:augNP} and Lemma~\ref{lem:augNP_canyon}).
	\item[{$E_{con}$}] top compact edge of $\widehat{\PP}(f_x,\cC)$, of co-slope $d_\cC$ (Theorem~\ref{newton_inv} and Lemma~\ref{lem:mu_equals_length}).
	\item[{$E_\delta$, $\omega$; $H_{\min}(\cC,\cC')$}] compact edge of $\PP(f_x,\gamma)$ of co-slope $\delta\in(\delta_{top},d_\cC)$ and the $Y$--intercept of its supporting line; minimal second-level exponent (Section~\ref{newton_interpret} and Corollary~\ref{cor:Hmin_equals_omega_plus_delta}).
	\item[{$\gamma_i$, $I(\cC)$}] Puiseux representatives of polars (the Puiseux roots of $f_x$); the indices $i$ with $\gamma_{i,*}\in\cC$ (Section~\ref{s:multiplicity}).
	\item[{$\mult(\cC)$, $\mult^{\mathrm{bar}}(\cC)$}] polar and bar multiplicity of $\cC$ (Definitions~\ref{def:PT-mult-canyon} and~\ref{def:bar-mult-canyon}).
	\item[{$\gamma_i^{<d}$, $\equiv_d$, $\operatorname{Tr}(\cC)$, $b_\cC$}] Puiseux trunk, the relation $\ord(\gamma_i-\gamma_j)\ge d$, trunk curve, number of trunks (Section~\ref{s:multiplicity} and Definition~\ref{def:number-trunks}).
	\item[{$K$, $\Horn^{(e)}(\alpha_*;\rho;\eta)$, $J^{(e)}(\alpha)$}] Gaussian curvature, horn domain, $e$-jet (Section~\ref{s:curvature}).
	\item[{$\mathcal{M}_f(\cC)$, $\mu_f(\cC)$}] total asymptotic Gaussian curvature concentrated in $\cC$; partial Milnor number of $\cC$ (Section~\ref{s:curvature} and Theorem~\ref{thm:KKP-C}).
\end{description}

\section*{Acknowledgements}
We would like to thank the anonymous referee for the careful reading of the manuscript and constructive suggestions.

This paper is part of a project initiated jointly with Mihai Tib\u ar, who sadly passed away before the completion of the manuscript. We gratefully acknowledge Mihai's outstanding contributions and remember him with affection both as a colleague and as a dear friend.

\section*{Statements and Declarations}

\noindent\textbf{Funding.} The authors acknowledge partial support from the grant ``Singularities and Applications'' -- CF 132/31.07.2023 funded by the European Union -- NextGenerationEU -- through Romania's National Recovery and Resilience Plan. The first named author is partially supported by the grant of Narodowe Centrum Nauki, number 2024/55/B/ST1/01412.

\end{document}